\documentclass[12pt]{amsart}
\usepackage{geometry}                
\usepackage{graphicx}
\usepackage{amssymb}
\usepackage{epstopdf}
\usepackage{amsmath}
\usepackage{color}
\usepackage{hyperref}
\usepackage{pdfsync}
 \usepackage{tikz}
 
\usepackage[all]{xy}
\xyoption{matrix}
\xyoption{arrow}

\usetikzlibrary{arrows,decorations.pathmorphing,backgrounds,positioning,fit,petri}

 \tikzset{help lines/.style={step=#1cm,very thin, color=gray},
help lines/.default=.5} 
\tikzset{thick grid/.style={step=#1cm,thick, color=gray},
thick grid/.default=1} 

\newtheorem{thm}{Theorem}[section]
\newtheorem{lem}[thm]{Lemma}
\newtheorem{cor}[thm]{Corollary}
\newtheorem{prop}[thm]{Proposition}

\theoremstyle{definition}
\newtheorem{defn}[thm]{Definition}
\newtheorem{eg}[thm]{Example}

\theoremstyle{remark}
\newtheorem{rem}[thm]{Remark}

\numberwithin{equation}{section}

\newcommand{\vs}[1]{\vskip .#1 cm} 

 \newcommand{\onto}{\twoheadrightarrow}

 \newcommand{\ot}{\leftarrow}

\DeclareMathOperator{\coker}{coker}
\DeclareMathOperator{\Hom}{Hom}%
\DeclareMathOperator{\undim}{\underline{dim}}
 
\newcommand{\field}[1]{\mathbb{#1}}
\newcommand{\ZZ}{\ensuremath{{\field{Z}}}}
\newcommand{\CC}{\ensuremath{{\field{C}}}}
\newcommand{\RR}{\ensuremath{{\field{R}}}}

\newcommand{\commentout}[1]{}

\newcommand{\cA}{\ensuremath{{\mathcal{A}}}}

\newcommand{\cB}{\ensuremath{{\mathcal{B}}}}

\newcommand{\cF}{\ensuremath{{\mathcal{F}}}}

\newcommand{\cG}{\ensuremath{{\mathcal{G}}}}

\newcommand{\cL}{\ensuremath{{\mathcal{L}}}}

\newcommand{\cS}{\ensuremath{{\mathcal{S}}}}

\newcommand{\cU}{\ensuremath{{\mathcal{U}}}}

\title{More ghost modules I}
\author{Kiyoshi Igusa}
\address{Department of Mathematics, Brandeis University, Waltham, MA 02454}\email{igusa@brandeis.edu}
\thanks{Supported by the Simons Foundation}



\begin{document}

\begin{abstract}{Ghost modules were introduced in \cite{GrInvRedo} without definitions or proofs. We also introduced stability diagrams or ``relative pictures'' for torsion classes and torsion-free classes for representations of Dynkin quivers. Modules which were not in the chosen class reappeared as ``ghosts'', in fact one missing module produced two ghosts. In this short paper, we give a precise definition of ghost modules. We give several examples and prove basic properties of ghosts and pictures for torsion and torsion-free classes. We also introduce a third kind of ghost which we call ``extension ghosts''. In the next paper we will explain how these new ghosts can be used to visualize the computation of other invariants of $K_3$ of group rings.
}
\end{abstract}

\maketitle

\tableofcontents

\section*{Introduction}

{In our paper \cite{GrInvRedo} we introduced ``ghost modules'' and gave heuristic explanations for what they are since we had only two examples (which came from Morse theory). One of the mysterious properties of ghosts was: We killed one module and it came back as two ghosts! Ghost $B$, which we now call $Z_b$, had a definition. But ghost $A$, renamed $Z_a$ below, did not have a good explanation, except that we knew from the intended application that it must be there. The ghost modules arose in the context of stability diagrams for torsion classes and torsion-free classes \cite{Apostolos-Idun}. The modules which were excluded from these classes came back as ``ghosts''.}

{In this paper, we describe more precisely what the ghosts are and where they appear. There are three kinds of ghosts which we call ``subobject ghosts'', ``quotient ghosts'' and ``extension ghosts''. The subobject ghosts appear naturally in torsion classes. Quotient ghosts appear in torsion-free classes. This makes sense since torsion classes already have all their quotients. But subobjects are missing, only their ghosts can exist. Similarly for torsion-free classes.}

{We also introduce a new kind of ghost which we call an ``extension ghost''. Whereas the subobject ghosts and quotient ghosts are related to the generalized Grassmann invariant \cite{GrInvRedo}, extension ghosts are related to the higher Reidemeister torsion from \cite{IgKlein}. These are algebraic K-theory invariants associated to families of diffeomorphisms of smooth manifold. This application will be explained in another paper with several coauthors \cite{MoreGhosts2}. The current paper will be purely algebraic.}

{We also give proofs for two of the theorems claimed without proof in \cite{GrInvRedo}: the relative Harder-Narasimhan (HN) stratification for classes of modules closed under extension and the maximal relative Hom-orthogonality of stable sequences of modules in such classes using a modified version of Bridgeland stability. The notion of ``relative'' (forward) Hom-orthogonality is critical to the appearance of ghosts. An ordered pair of objects in our class of modules is called ``relatively Hom-orthogonal'' if there is no nonzero admissible morphism from the first to the second where a morphism is admissible if its kernel, image and cokernel are in our class of modules. For example, suppose we have an exact sequence $A\to B\to C$ where $B,C$ are in our class and $A$ is not. If there were no other morphism from $B$ to $C$, then the pair $B,C$ would be relatively Hom-orthogonal, since $A$ is missing and even the ghost of $A$ would not be allowed. However, if $C$ came before $B$, we could allow the ghost of $A$ to appear. So, $C,B$ followed by the ghost of $A$ would be a relatively Hom-orthogonal sequence with ghosts.}

{We give precise definitions and statements for very general classes of modules in Section \ref{sec1: definitions}. We construct a relative version of the ``wall and chamber'' structure for our class and define a maximal green sequence (MGS) to be a sequence of wall crossings (the standard definition, but with modified walls: Definition \ref{def: D(M)}) and show that they give a relatively Hom-orthogonal sequence. Wall and chamber structures originate in work of Br\"ustle, Smith and Treffinger \cite{BSTwallandchamb}. There is also work of Asai \cite{Asai} and this author \cite{Linearity} which we find useful. The notion of a Hom-orthogonal sequence was introduced in \cite{Linearity} and shown to be equivalent to a MGS in some cases in \cite{MGS4clustertilted}. This theorem was extended to all finite dimensional algebras by Demonet in \cite{KellerSurvey}. This paper by Keller and Demonet also gives a very nice survey of MGSs.
}

{
There are some pathologies in the general case. In Remark \ref{rem: why list of bifurcations is complete} we note that one case of the bifurcation of subobject ghosts is ``pathological'' and does not occur in a torsion class (or torsion-free class). In Section \ref{ss: extension ghosts}, the ``relative picture'' displayed in Figure \ref{Figure11} does not satisfy the definition of picture given in \cite{IT14}, \cite{ITW}. Figure \ref{Figure13} shows the problem and how to fix it using ``extension clones'' which we will need in the next paper.
}

{In Section \ref{ss: section 2}, we add the assumption that our class of modules is closed under extension. The purpose of this is to treat torsion classes and torsion-free classes simultaneously and to avoid the pathologies. This assumption is sufficient to prove the relative version of HN-stratification and show that MGSs give a maximal relative Hom-orthogonal sequence. The HN-stratification and Hom-orthogonality of maximal green sequences in abelian length categories has been studied in detail by many authors, e.g. \cite{BSTstability}, \cite{KellerSurvey} and Liu and Li \cite{FangLiLiu}. The novelty here is that we modify the definitions so that they apply to classes of modules closed under extension such as torsion classes and torsion-free classes. In applications we will need to consider more general classes of modules such at the one in Figure \ref{Figure11} and use ``clones'' such as in Figure \ref{Figure13}.
}

{
Section \ref{sec 3: ghosts} is about ghosts. In subsection \ref{ss: subsection on subobject ghosts} we give a detailed description of subobject ghosts and their bifurcation. When the domain of one ghost $Z_b$ hits a ``splitting wall'' it bifurcates by splitting off another ghost $Z_a$. The splitting wall is a boundary for $Z_a$, but $Z_b$ continues through the wall. We present the pictures and give more details about the five main examples of subobject ghosts and their bifurcation in subsection \ref{ss: examples}. In subsection \ref{ss: quotient object ghosts} we give the definitions and list the examples of quotient ghosts. There is a simple duality which relates subobject ghosts and quotient ghosts (see Example \ref{eg: example of duality}), so we don't go into too much detail about the quotient ghosts. In subsection \ref{ss: extension ghosts} we give a definition and two examples of extension ghosts.
}

{The next paper \cite{MoreGhosts2} will concentrate more on pictures coming from Morse theory and their relation to algebraic K-theory.
}

\section{Basic definitions}\label{sec1: definitions}

{Let $\Lambda$ be a finite dimensional algebra of rank $n$ (having $n$ simple modules) and let $\cL_0$ be a finite set of $\Lambda$-bricks $B_i$ whose dimension vectors $\undim B_i\in\ZZ^n$ are pairwise linear independent. Thus, they have distinct orthogonal hyperplane $H(B_i):=\{\theta\in\RR^n\,:\, \theta(B_i)=0\}$, where $\theta(M)$ denotes the dot product of $\theta$ with $\undim M$. Let $\cL=add\,\cL_0$ be the full subcategory of $mod\text-\Lambda$ of all finite direct sums of these bricks.}

\begin{defn}\label{def: admissible morphism}
An \emph{admissible} morphism between two objects of $\cL$ is defined to be a morphism $f:X\to Y$ image, kernel and cokernel are in $\cL$. A \emph{weakly admissible} morphism is one whose image and cokernel lie in $\cL$ but whose kernel is only required to lie in $\cF ilt(\cL)$, i.e., $\ker f$ has a filtration whose subquotients lie in $\cL$. Admissible morphisms are usually not closed under composition, but any composition of weakly admissible epimorphisms will be weakly admissible.
\end{defn}

In Section 2 we will need to assume that $\cL$ is closed under extension (in particular, we will need to assume its bricks are rigid). In this section we will see what we can do with no such assumptions. Here is a simple preview example.

{
\begin{eg}\label{eg: Kroneker}
Let $\Lambda$ be the Kroneker algebra. This is the tame hereditary algebra over some field $K$ given by the quiver with two vertices 1,2 and two arrows from 2 to 1. Let $\cL_0$ consist of three indecomposable modules, $P_1,P_2$ and $M$ where $M$ is one of the modules with dimension vector $(1,1)$. These are bricks and they form a short exact sequence
\[
	P_1\xrightarrow f P_2\xrightarrow g M.
\]
Let $\cL=add\,\cL_0$. Then $f,g$ are the only weakly admissible morphisms between indecomposable objects of $\cL$. The monomorphism $P_1\to M$ is not admissible since the quotient $S_2$ is not in $\cL$. However, we consider the exact sequence
\[
	P_1\to M\to S_2
\]
as a ghost of the missing module $S_2$. We call this a ``quotient ghost'' since $S_2=M/P_1$ is a quotient of bricks in $\cL$. There are no other quotient ghosts or subobject ghosts. For example $P_1^2\to P_2\to S_2$ is not another ghost of $S_2$ since the kernel $P_1^2$ is not a brick and thus not allowed in a ghost. This is shown in Figure \ref{fig: Kroneker ghosts} where we also see the extension ghost of $P_2$. We use the notation $D(M)$ from Definitionn \ref{def: D(M)} below.
\end{eg}
}

{
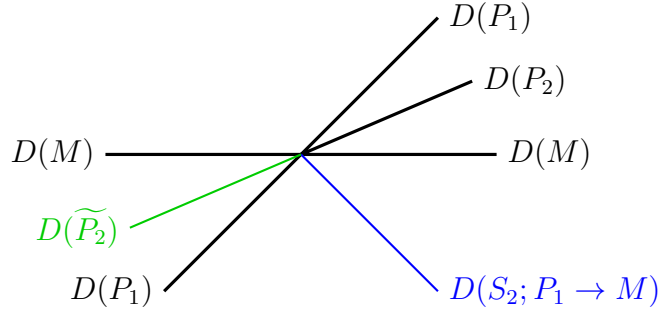
\begin{figure}[htbp]
\begin{center}
\begin{tikzpicture}[scale=1.3]
\draw[very thick] (-2,0)--(2,0) (-1.4,-1.4)--(1.4,1.4) (0,0)--(1.75,.75);
\draw[ thick,blue] (0,0)--(1.4,-1.4);
\draw[ thick,green!80!black] (0,0)--(-1.75,-.75);
\draw (1.4,1.4) node[right]{$D(P_1)$};
\draw (-1.4,-1.4) node[left]{$D(P_1)$};
\draw (2,0) node[right]{$D(M)$};
\draw (-2,0) node[left]{$D(M)$};
\draw (1.75,.75) node[right]{$D(P_2)$};
\draw[green!80!black] (-1.75,-.75) node[left]{$D(\widetilde{P_2})$};
\draw[blue] (1.4,-1.4) node[right]{$D(S_2;P_1\to M)$};
\end{tikzpicture}
\caption{This is the picture for the Kroneker Example \ref{eg: Kroneker}. Since $P_1$ and $M$ are minimal in $\cL$, $D(P_1)$ and $D(M)$ are hyperplanes. Since $M$ is a weakly admissible quotient of $P_2$, the domain $D(P_2)$ is given by $\theta(P_2)=0$ and $\theta(M)\ge0$. Therefore $D(P_2)$ contains only that part of the hyperplane $H(P_2)$ which is above $D(M)$. The missing part of $D(P_2)$ is the extension ghost $D(\widetilde{P_2})$. The domain of the quotient ghost $\cG h^\ast(S_2;P_1\to M)$ is given by $\theta(S_2)=0$ and $\theta(M)\le0$. So, the ghost of $S_2$ is below $D(M)$. }
\label{fig: Kroneker ghosts}
\end{center}
\end{figure}
}

\subsection{Walls and chambers}\label{sec 1.1: Walls and chambers}

{
Stability conditions were introduced by King \cite{King} and later generalized by Bridgeland \cite{Bridgeland}, \cite{Bridgeland2}. King's original definition was to make, for any module $M$, $D_K(M)$ equal to the set of all linear ``stability conditions'' $\theta:\RR^n\to \RR$ so that $\theta(M)$, defined to be $\theta$ applied to the dimension vector of $M$ is equal to zero and $\theta(M')\le0$ for all $M'\subset M$. This subobject condition is equivalent to the condition $\theta(M'')\ge0$ for all quotient objects $M''$ of $M$. In our setting, these two notions are not equivalent and we take the second.
}

{\begin{defn}\label{def: D(M)}
For $M$ a brick, $D(M)$ is defined to be the set of all $\theta$ so that $\theta(M)=0$ and $\theta(M')\ge0$ for all weakly admissible quotients $M'$ of $M$. This is a closed set. Since $\cL$ has only finitely many bricks, the union of all $D(M)$ is closed. The \emph{interior} $int\,D(M)\subset D(M)$ is the set of all $\theta$ so that $\theta(M)=0$ and $\theta(M')>0$ for all weakly admissible quotients $M'$ of $M$. The \emph{boundary} $\partial D(M)$ is the complement in $D(M)$ of its interior $int\,D(M)$. We sometimes refer to $D(M)$ as a \emph{wall}. (See \cite{BSTwallandchamb}.) The union of the walls $D(M)$ is called the \emph{(relative) picture} for $\cL$. It is also known as the (relative) ``stability diagram'' or ``scattering diagram''.
\end{defn}

We say that $M$ is \emph{minimal} if it has no weakly admissible quotients. In that case $D(M)=H(M)$ is the hyperplane given by the single linear equation $\theta(M)=0$.
}

\begin{prop}\label{prop: U D(M)=U int D(M)}
The union of $int\,D(M)$ is equal to $\bigcup D(M)$. 
\end{prop}

\begin{proof} Let $\theta\in\partial D(M)$. Then $\theta(M')=0$ for some weakly admissible quotient $M'$ of $M$. Then $\theta\in D(M')$ since any weakly admissible quotient $M''$ of $M'$ is also a weakly admissible quotient of $M$. So, $\theta(M'')\ge0$. By induction on the size of $M$, we have that the wall $D(M')$ lies in the union of interiors of other walls.
\end{proof}

The components of the complement of $\bigcup D(M)$ in $\RR^n$ are called \emph{chambers} \cite{BSTwallandchamb}. To each chamber $\cU$ we will associate a set of objects in $\cL$ which we will denote $\cS(\cU)$.

\begin{defn}\label{def: the set S(th)}
For each $\theta\in \RR^n$, let $\cS(\theta)$ be the set consisting of $0$ and all nonzero $M\in \cL$ so that $\theta(M)>0$ and $\theta(M')>0$ for all weakly admissible quotients $M'$ of $M$. Thus, if $M$ is in $\cS(\theta)$, then any weakly admissible quotient of $M$ is also in $\cS(\theta)$.
\end{defn}

For example, if $\eta$ is the vector all of whose coordinates are 1, then $\cS(\eta)=\cL$ and $\cS(-\eta)=0$ since $\eta(M)=\dim M>0$ for all nonzero $M$.

\begin{lem}\label{lem1 for sets S(U)}
For any $X\in\cL$ and $\theta$ not on any wall $D(M)$, $X$ is not in $\cS(\theta)$ if and only if $X$ has an indecomposable weakly admissible quotient $B$ so that $\theta(B)<0$.
\end{lem}

\begin{proof} Clearly, the existence of such a $B$ implies $X$ is not in $\cS(\theta)$. So suppose $X$ is not in $\cS(\theta)$. Then, by definition, there exists $X'$, a weekly admissible quotient of $X$ (possibly $X'=X$) so that $\theta(X')\le0$. Take $X'$ minimal. Then $X'$ must be indecomposable, otherwise the components $X_i$ of $X'$, being smaller weakly admissible quotient of $X$, would have $\theta_0(X_i)>0$ making $\theta(X')>0$. We claim that $\theta(X')<0$ since, otherwise, $\theta(X')=0$ and $\theta(X'')>0$ for all weakly admissible quotients $X''$ of $X'$, making $\theta$ an element of $int\,D(X')$ contrary to assumptions.
\end{proof}

\begin{lem}\label{Lemma 2} 
$X\in \cS(\theta)$ if and only if, for every indecomposable weakly admissible quotient $B$ of $X$, $\theta(B)>0$. Moreover, $B\in\cS(\theta)$ for any such $B$.
\end{lem}
\vs1
\begin{proof} This condition is necessary by definition of $\cS(\theta)$. Conversely, if $X\notin \cS(\theta)$, the condition fails by the previous lemma.
\end{proof}

\begin{thm}\label{thm: S(U) is constant}
The set $\cS(\theta)$ is locally constant on the complement of the set of walls, $D(M)$. Thus it is constant on each component of this set, i.e., in each chamber $\cU$.
\end{thm}

\begin{proof} Take $\theta_0$ not on any wall $D(B)$. Let $\cA$ be the set of all bricks $B_i$ so that $\theta_0(B_i)>0$ and let $\cB$ be the set of all bricks $B_j$ so that $\theta_0(B_j)<0$. Let $U$ be the set of all $\theta$ so that $\theta(B_i)>0$ for all $B_i\in\cA$ and $\theta(B_j)<0$ for all $B_j\in\cB$. Since $\cL$ has only finitely many bricks, $U$ is an open set containing $\theta_0$. 

By Lemma \ref{lem1 for sets S(U)}, $X\notin \cS(\theta_0)$ if and only if $X\onto B$ for some $B\in \cB$. Then $X\notin\cS(\theta)$ for all $\theta\in U$ since $\theta(B)<0$. So, $\cS(\theta)\subset\cS(\theta_0)$.

Let $X\in\cS(\theta_0)$. Then Lemma \ref{Lemma 2} implies that for any $X\onto B$, $B\in\cA$. Since $\theta(B)>0$ for all $\theta\in U$, we have $X\in\cS(\theta)$. So, $\cS(\theta_0)\subset\cS(\theta)$. So, they are equal and $\cS(\theta)$ is locally constant.
\end{proof}

For any chamber $\cU$, we denote by $\cS(\cU)$ the set $\cS(\theta)$ for any and thus all $\theta\in\cU$. We would like to see how $\cS(\theta)$ changes as $\theta$ crosses a wall.

\subsection{Wall crossing}\label{sec 1.2: wall crossing}

Suppose two chambers $\cU$ and $\cU'$ are separated by a wall $int\,D(M)$. Let $\theta_0$ be a generic point on the portion of this wall which separates $\cU,\cU'$. Thus $\theta_-=\theta_0-\varepsilon \eta\in\cU$ for sufficiently small $\varepsilon>0$ and $\theta_+=\theta_0+\varepsilon \eta\in\cU'$ for sufficiently small $\varepsilon>0$ where $\eta=(1,1,\cdots,1)$ is the vector so that $\eta(X)>0$ for all $X\neq0$.

\begin{lem}\label{lem: M on one side of wall D(M)} $M\in\cS(\theta_+)$ but $M$ is not in $\cS(\theta_0)$ or in $\cS(\theta_-)$.
\end{lem}

\begin{proof} Since $\theta_0(M)=0$ and $\theta_-(M)<0$, $M$ is not in $\cS(\theta_0)$ or $\cS(\theta_-)$. On the other hand, $\theta_+(M')>\theta_0(M')\ge0$ for all weakly admissible quotients $M'$ of $M$. So, $M\in\cS(\theta_+)$.
\end{proof}

\begin{thm}\label{thm: S(t0)=S(t-)}
$\cS(\cU)=\cS(\theta_-)=\cS(\theta_0)\subsetneq \cS(\theta_+)=\cS(\cU')$.
\end{thm}

\begin{proof} 
Since $X\in \cS(\theta)$ is an open condition on $\theta$, $\cS(\theta_0)\subset \cS(\theta_-)$ and $\cS(\theta_0)\subset \cS(\theta_+)$. Now consider $X\in\cS(\theta_-)$. Then, for $X'$ equal to $X$ or any weakly admissible quotient of $X$ we have $\theta_-(X')>0$. But $\theta_0(X')>\theta_-(X')$. So, $X\in \cS(\theta_0)$. So, $\cS(\theta_-)=\cS(\theta_0)$. Since $M\in\cS(\theta_+)$ but $M\notin\cS(\theta_0)$ we have $\cS(\theta_0)\subsetneq \cS(\theta_+)$.
\end{proof}

\begin{lem}\label{lem: no morphisms X to M}
Let $\cU,\cU'$ be as above. They are separated by the wall $D(M)$. Then, for any $X\in\cS(\cU)$, there do not exist any nonzero weakly admissible morphisms $X\to M$. Conversely, if $X\notin \cS(\cU)$ but $X\in\cS(\cU')$, then there is a nonzero weakly admissible epimorphism $X\onto M$.
\end{lem}

\begin{proof} Take $\theta_0\in D(M)$ generic as before. So, $\cS(\theta_0)=\cS(\cU)$. Let $f:X\to M$ be a nonzero weakly admissible morphism with image $Y$. Since $X\in \cS(\theta_0)$ and $Y$ is a weakly admissible quotient of $X$, we have $\theta_0(Y)>0$ making $\theta_0(M/Y)<0$. This contradicts the assumption that $\theta_0\in D(M)$. Therefore $f:X\to M$ does not exist.

Conversely, suppose $X\notin \cS(\cU)$ but $X\in\cS(\cU')$. By Lemma \ref{lem1 for sets S(U)}, there exists a weakly admissible quotient $X'$ of $X$ so that $\theta_-(X')<0$. Since there are only finitely many such $X'$, there is one choice of $X'$ so that $\theta_i(X')<0$ for a sequence $\theta_i$ converging to $\theta_0$. Since $\theta_+(X')>0$, we must have $\theta_0(X')=0$ and $\theta_0(X'')\ge0$ for all weakly admissible quotients $X''$ of $X'$. Thus $\theta_0\in D(X')$. Since $\theta_0\in int\,D(M)$ is generic, it does not lie on any other wall. So, $X'=M$ as claimed.
\end{proof}

{%
\begin{figure}[htbp]
\begin{center}

\begin{tikzpicture}[scale=1.4] 
%
{
\begin{scope}
\begin{scope}
\clip rectangle (2,1.3) rectangle (0,-1.3);
\draw[thick] (0,0) ellipse [x radius=1.5cm,y radius=1.21cm];
\end{scope}
\begin{scope}[xshift=.87mm]
		\draw[thick] (1.8,.6) node[left]{\small$D(P_3)$};
\end{scope}
\begin{scope}[xshift=-.7cm]
	\draw[thick] (0,0) circle[radius=1.4cm];
		\draw (-1,1.1) node[left]{\small$D(S_1)$};
\end{scope}
\begin{scope}[xshift=.7cm]
	\draw[thick] (0,0) circle[radius=1.4cm];
		\draw (1,1.1) node[right]{\small$D(I_2)$};
\end{scope}
\begin{scope}[yshift=-1.35cm]
	\draw[thick] (0,0) circle[radius=1.4cm];
	\draw (1.15,-.8) node[right]{\small$D(S_3)$};
\end{scope}
\begin{scope}[yshift=-1cm]
\end{scope}
\coordinate(B1) at (1.39,-1.22);
\coordinate(B2) at (-.69,-.14);
\coordinate(B) at (-.1,-.9);
\coordinate(A) at (1.24,-.69);
\coordinate(C) at (.48,-.75);
\coordinate(AC) at (.84,-.73);
\end{scope} 
\begin{scope}[xshift=4cm,yshift=-.5cm]
\draw (0,0) node{$S_1$}; 
\draw(0,0)circle[radius=2.4mm];
\draw[red!80!black] (0.5,0.6) node{$P_2$}; 
\draw (1.5,0.6) node{$I_2$};
\draw(1.5,0.6)circle[radius=2.4mm];
\draw (1,1.2) node{$P_3$};
\draw(1,1.2)circle[radius=2.4mm];
\draw[blue] (1,0) node{$S_2$}; 
\draw (2,0) node{$S_3$}; 
\draw(2,0)circle[radius=2.4mm];
\draw[->] (0.13,0.15)--(.33,.4);
\draw[->] (1.13,0.15)--(1.33,0.4);
\draw[->] (0.63,0.75)--(.83,1);
\draw[->] (0.63,0.37)--(.83,.14);
\draw[->] (1.63,.37)--(1.83,.14);
\draw[->] (1.13,.97)--(1.33,.74);
\end{scope}

\coordinate(U1)at(-1.4,0.2);
\coordinate(U2)at(0,0.4);
\coordinate(U3)at(1.1,0.15);
\coordinate(U4)at(1.8,0);
\coordinate(U5)at(-.8,-.9);
\coordinate(U6)at(0,-.5);
\coordinate(U7)at(.9,-.7);
\coordinate(U8)at(1.1,-1.1);
\coordinate(U9)at(0,-2);
\draw (U1) node{$S_1$};
\draw (U2) node{$S_1I_2P_3$};
\draw (U3) node{$I_2P_3$};
\draw (U4) node{$I_2$};
\draw (U5) node{$S_1S_3$};
\draw (U6) node{$\cL$};
\draw (U7) node{\tiny$S_3I_2P_3$};
\draw (U8) node{\small$S_3I_2$};
\draw (U9) node{$S_3$};
\draw[green!80!black,->](-3,-.5)--(-.2,-.5);
\draw[green!80!black](-2.8,-.5) node[below]{$\theta_t$};
}
\end{tikzpicture}

\caption{In each chamber $\cU$, the set of bricks in $\cS(\cU)$ is indicated. Usually, crossing a wall $D(M)$ will add just $M$ to $\cS(\cU)$. But sometimes we get more. For example, on the left side, when crossing the wall $D(I_2)$, two modules are added. This is because, if we go from the right, we cross three walls. The quiver is $1\ot 2\ot 3$. The set of bricks in $\cL$ is $S_1,P_3,I_2,S_3$. These are circled in the AR quiver. The figure show the ``picture'' for this set which is defined to be the stereographic projection of the intersection of these walls with the unit sphere $S^2$ in $\RR^3$. The green path $\theta_t$ is explain below.}
\label{Figure00}
\end{center}
\end{figure}
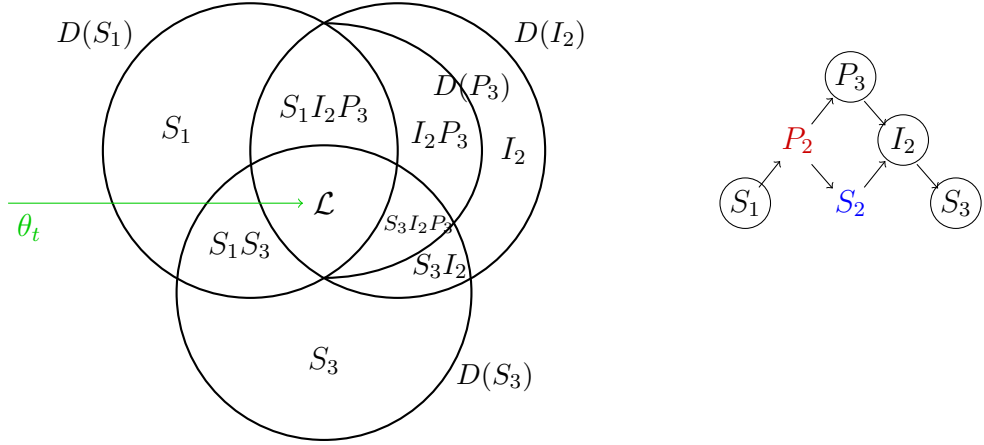
}

\subsection{Maximal green sequences}\label{sec 1.3: MGS}

We define a \emph{green path} to be a generic smooth path $\theta_t$ in $\RR^n$ which crosses walls in the positive direction. In other words, whenever $\theta_{t_0}\in int\,D(M)$, $\theta_t$ crosses from the negative side of $D(M)$ to its positive side. Equivalently,
\[
	\frac{d}{dt}\theta_t|_{t=t_0}\cdot \undim M=\frac{d}{dt}\theta_t(M)|_{t=t_0}>0.
\]
Thus $\theta_t$ crosses a sequence of walls $D(M_1), \cdots, D(M_m)$ in their interiors and these walls separate chambers $\cU_0, \cdots,\cU_m$ which the path goes through. We say the green path is \emph{maximal} if $\cU_0$ contains a point with all coordinates negative (so that $\cS(\cU_0)=0$) and $\cU_m$ contains a point with all positive coordinates (so $\cS(\cU_m)=\cL$). For a maximal green path, the resulting sequence of bricks, $M_1,\cdots,M_m$ will be called a \emph{maximal green sequence (MGS)} for $\cL$.

{In Figure \ref{Figure00}, the indicated green path gives the MGS $S_1,S_3,I_2$. According to the following theorem, the module $P_3$ is missing because it would need to come before $S_1$ and after $I_2$. For example, if we come from the right we get the MGS $I_2,S_3,P_3,S_1$. The object $I_2$ is allowed to come before $S_3$ since the map $I_2\to S_3$ is not admissible.
}
\begin{thm}\label{thm: MGS implies relatively hom orthogonal}
For any MGS $M_1,\cdots,M_m$, there are no nonzero weakly admissible morphisms $M_i\to M_j$ for $i<j$.
\end{thm}

\begin{proof} Let $\cU_i$ be the chamber containing the portion of the maximal green path between $D(M_i)$ and $D(M_{i+1})$. Then $M_i\in\cS(\cU_i)$ and $\cS(\cU_i)\subset \cS(\cU_{j-1})$. So, by Lemma \ref{lem: no morphisms X to M} above, there are no nonzero weakly admissible morphisms $M_i\to M_j$.
\end{proof}

\begin{defn}\label{def: linear MGS}
By a \emph{smooth green path} or \emph{smooth Bridgeland stability condition} we mean a smooth path $\gamma_t$ in $\RR^n$ having the following properties where we use the same notation for $\gamma_t$ as for $\theta_t$.
\begin{enumerate}
\item For each brick $M$, $\frac d{dt}\gamma_t(M)$ is positive for all $t$.
\item For $t<<0$, the coordinates of $\gamma_t$ are all negative (so that $\gamma_t$ is on the negative side of all hyperplanes $H(M)$.
\item For $t>>0$, $\gamma_t$ is on the positive side of each hyperplane $H(M)$.
\item When objects $Z$ not in $\cL$ are brought back as ``ghosts'', we will also assume condition (1) for $M=Z$.
\end{enumerate}
We say that $\gamma_t$ is a \emph{linear green path} if it is a linear function of $t$, i.e., $\gamma_t={\bf h}+t{\bf k}$ for ${\bf h,k}\in \RR^n$ where, by (2) and (3), all of the coordinates of $\bf k$ are positive.
\end{defn}

{A linear green path is also called a \emph{Bridgeland stability condition} since the complex vector $i{\bf k}-{\bf h}\in \CC^n$ gives a ``central charge'' $Z:K_0(\Lambda)\to\CC$ and a module $M$ is ``stable'' with respect to this central charge if and only if the corresponding linear green path passes through the interior of the domain $D(M)$ of $M$ \cite{Bridgeland}, \cite{Bridgeland2}, \cite{Linearity}, \cite{GrInvRedo}, \cite{FangLiLiu}.}

We will assume $\gamma_t$ is generic so that it does not cross two hyperplanes at the same time. So, if $\gamma_t$ passes through $H(M)$ at time $t_M$, these times are distinct. If $M$ is not necessarily indecomposable, say $M=\bigoplus M_i$, let $t_M$ be the unique time so that
\[
	\gamma_{t_M}(M)=0.
\]
Since $\gamma_t$ is a linear function, $\gamma_{t_M}(M)=\sum\gamma_{t_M}(M_i)$.
Then we have the following \emph{relative Bridgeland stability condition}.

\begin{defn}\label{def: linear stability condition}
Given a generic smooth green path $\gamma_t$, we say that $M$ is \emph{relatively stable} or \emph{$\gamma_t$-stable} if $t_M>t_Y$ for all proper weakly admissible quotients $Y$ of $M$.
\end{defn}

\begin{thm}\label{thm: stable implies path crosses D(M)}
The object $M\in\cL$ is {$\gamma_t$-stable} with respect to a generic smooth green path $\gamma_t$ if and only if the path goes through $D(M)$, i.e., $\gamma_{t_M}\in D(M)$. Furthermore, the path goes through the interior of $D(M)$ in that case.
\end{thm}

\begin{proof}
Suppose that $\gamma_t$ passes though $D(M)$ at time $t_M$. Then, for any weakly admissible quotient $Y$ of $M$, we have $\gamma_{t_M}(Y)\ge0$. Since $\gamma_t(Y)$ is a strictly increasing function of $t$, we must have $t_Y\le t_M$. This is also true for any component $Y_i$ of $Y$. When $Y\neq M$, this implies $t_{Y_i}<t_M$. And this implies $t_Y<t_M$. Thus $M$ is {$\gamma_t$-stable}.

Conversely, suppose $M$ is {$\gamma_t$-stable}. Then $t_Y<t_M$ for all proper weakly admissible quotients $Y$ of $M$. So, $\gamma_{t_M}(Y)>0$ and $\gamma_{t_M}(M)=0$. So, $\gamma_{t_M}$ lies in $int\,D(M)$.
\end{proof}

{
The corresponding statement holds for any maximal green path $\theta_t$ although this is not so surprising since we are already assuming the path goes through the walls $D(M_i)$. 

\begin{prop}
Take a maximum green path $\theta_t$. Assume it is generic so that it does not cross two hyperplanes $H(X)$ at the same time and it crosses each $H(X)$ transversely. For each brick $B$ let $t_B$ be the last time the path crosses the hyperplane $H(B)$. Then $B$ is equal to $M_j$ and thus $\gamma_{t_B}\in D(M_j)$ for some $j$ if and only if $t_X<t_B$ for all weakly admissible quotients $X$ of $B$.
\end{prop}

\begin{proof} Suppose first that $B=M_j$ for some $j$. Then the path crosses $D(B)$ at a certain time say $t_0$ and, since the path will be traveling in chambers that come after $D(B)$, $\theta_t$ will be positive on $B$ and all of it weakly admissible quotients for $t>t_0$. Therefore $t_0=t_B$ and $t_0>t_X$ for all weakly admissible quotients $X$ of $B$.

Conversely, suppose that $t_X<t_B$ for all weakly admissible quotients $X$ of $B$. Then, at $t=t_B$, we will have $\theta_t(X)>0$ for all weakly admissible quotients $X$ of $B$. This implies that $\theta_{t_B}$ lies on the wall $D(B)$. Therefore $B=M_j$ for some $j$.
\end{proof}

The modules $M_i$ which occur in the sequence of walls $D(M_i)$ passed through by a linear green path will be called a \emph{linear maximal green sequence}. Liu and Li \cite{FangLiLiu} have given criteria to determine when a MGS is linear. We expect that the analogous statement will hold in our case.
}

\section{Theorems in extension closed case}\label{ss: section 2}

In this section we will assume that $\cL$ is closed under extension. To maintain the property that there are only finitely many indecomposable objects and that they are all bricks, we need to assume the bricks are rigid (otherwise a self-extension will be indecomposable but not a brick). One consequence is that all weakly admissible morphisms are admissible since $\cF ilt(\cL)=\cL$. So, we can drop the word ``weakly''.

{
\begin{thm}\label{HN-stratification} If $\cL$ is closed under extension, any MGS
$M_1,\cdots,M_m$ gives an HN-stratification of $\cL$, i.e., any object $X$ of $\cL$ has a filtration
\[
	0=X_0\subset X_1\subset\cdots\subset X_m=X
\]
so that each $X_i/X_{i-1}\in add\,M_i$.
\end{thm}

\begin{proof} We may assume $X$ is indecomposable.
Take $k$ minimal so that $X\in\cS(\cU_k)$. Then, by Lemma \ref{lem: no morphisms X to M}, there is an admissible epimorphism $f:X\onto M_k$. We claim that $Z=\ker f$ also lies in $\cS(\cU_k)$.

To show that $Z\in \cS(\cU_k)$, choose a sequence $\theta_i$ in $\cU_k$ converging to $\theta_0\in D(M_k)$. Since $\theta_i(X)>0$ for all $i$, we obtain in the limit that $\theta_0(X)\ge0$. Since $\theta_0(M_k)=0$, this implies $\theta_0(Z)=\theta_0(X)\ge0$ and this implies that $\theta_i(Z)>0$ for all $i$. Similarly, let $Z/A$ be an admissible quotient of $Z$. Then $X/A$, being an extension of $Z/A$ by $M_k$, also lies in $\cL$ and is an admissible quotient of $X$. So, $\theta_i(X/A)>0$ for all $i$ making $\theta_0(X/A)\ge0$. So, $\theta_0(Z/A)=\theta_0(X/A)\ge0$ which forces $\theta_i(Z/A)>0$ for all $i$. So, $Z$ lies in $\cS(\theta_i)=\cS(\cU_k)$ for all $i$. This implies that each component $Z_i$ of $Z$ lies in $\cS(\cU_k)$.

For each components $Z_i$ of $Z$, let $j$ be maximal so that $Z_i\in \cU_j$. Then $j\le k$ and $M_j$ is an admissible quotient of $Z_i$. Take the kernel of $Z_i\onto M_j$ and repeat this process. Then we obtain an HN-filtration of $X$.
\end{proof}

\begin{cor}\label{cor: HN is minimal}
A maximal green sequence $M_1,\cdots,M_m$ gives a minimal HN-stratification of $\cL$ in the sense that, if any term $M_i$ is deleted, the remainder is not an HN-stratification of $\cL$.
\end{cor}

\begin{proof} Suppose $M_i$ were not necessary. Then we could take an HN-filtration of $M_i$ using the other terms. This will have top $M_k$ and bottom $M_j$ where $j<k$. Since $M_k$ is an admissible quotient of $M_i$ and $M_j$ is an admissible subobject of $M_i$, we must have $k<i<j<k$, a contradiction.
\end{proof}

A similar argument shows the following.

\begin{defn}\label{def: relatively Hom orthogonal}
We say that a sequence of objects $M_1,\cdots,M_m$ in $\cL$ is \emph{relatively Hom-orthogonal} if there are no nonzero admissible morphisms $M_j\to M_k$ for $j<k$. A MGS satisfies this condition by Theorem \ref{thm: MGS implies relatively hom orthogonal}.
\end{defn}

\begin{cor}
Any MGS $M_1,\cdots,M_m$ is maximally relatively Hom-orthogonal in the sense that it cannot be extended by inserting an object of $\cL$ somewhere in the sequence.
\end{cor}

\begin{proof} 
Suppose we try to insert a new indecomposable object $X$, not equal to any $M_i$ into this sequence. Take the HN-filtration of $X$. This gives an admissible subobject $M_j$ of $X$ and an admissible quotient object $M_k$ of $X$ where $j<k$. In order to maintain relative Hom-orthogonality, we would need to insert $X$ before $M_j$ and after $M_k$. This is impossible since $j<k$. So, no new term $X$ can be inserted. So, the MGS is maximally relatively Hom-orthogonal.
\end{proof}
}

{
We were not able to prove the converse of these theorems, namely that all maximal relatively Hom-orthogonal sequences and minimal HN-stratifications are MGS's. Maybe more conditions on $\cL$ are needed. One important statement which we have not yet done is to verify that the sets $\cS(\cU)$ are distinct for different chambers $\cU$.
}

\begin{lem}\label{lem: existence of admissible subobject}
Suppose that $\theta(M)>0$ but $M\notin\cS(\theta)$. Then, $M$ has an admissible subobject $X$ which lies in $\cS(\theta)$.
\end{lem}

\begin{proof}
Since $M\notin\cS(\theta)$, it must have an admissible subobject $X$ so that $\theta(M/X)\le0$. Take the smallest such $X$. Then $\theta(X)=\theta(M)-\theta(M/X)>0$. For any admissible subobject $K$ of $X$, $K$ will be an admissible subobject of $M$ since $M/K$ is an extension of $X/K$ by $M/X$. By minimality of $X$ we have $\theta(M/K)>0$. So 
\[
	\theta(X/K)=\theta(M/K)-\theta(M/X)>0.
\] Since $X/K$ is a general admissible quotient of $X$, this implies $X\in\cS(\theta)$.
\end{proof}

\begin{thm}\label{thm: S(U) is not S(U')}
Each chamber $\cU$ is convex and, for distinct chambers $\cU,\cU'$, we have $\cS(\cU)\neq \cS(\cU')$.
\end{thm}

{
\begin{proof}
Suppose that $D(M_i)$ are the walls of $\cU$ and take $\varepsilon_i=\pm$ the sign telling on which side of $D(M_i)$ the chamber $\cU$ lies. Thus $M_i\in\cS(\cU)$ for $\varepsilon_i=(+)$ and $M_i\notin\cS(\cU)$ for $\varepsilon_i=(-)$. Then we claim:
\begin{enumerate}
\item[(a)] If $\theta$ lies on the other side of one of the hyperplanes $H(M_i)$ then $\cS(\theta)\neq \cS(\cU)$.
\item[(b)] $\cU$ is the region bounded by the hyperplanes $H(M_i)$, i.e., $\cU$ is the set of all $\theta$ so that $\varepsilon_i\theta(M_i)>0$. [This clearly implies $\cU$ is convex.]
\end{enumerate}
First, we note that (a) implies (b). Condition (a) implies that $\cU$ does not contain any points of the wrong side of any hyperplane $H(M_i)$. Therefore, $\cU$ is contained in the convex region bounded by these hyperplanes. Furthermore, this convex region must be equal to $\cU$, otherwise, there is a point $\theta_0\in \cU$ and another point $\theta_1\notin\cU$ which are both in the convex region bounded by the hyperplanes. Take the straight line connecting $\theta_0$ to $\theta_1$. This path cannot leave $\cU$ since it does not cross any of the walls that bound $\cU$. So, both lie in $\cU$. This proves (b) assuming (a).

Next we note that (a) and (b) together imply that $\cS(\cU)\neq\cS(\cU')$ for $\cU\neq \cU'$. By (b), $\cU'$ must contain some point $\theta$ on the wrong side of one of the walls. Then (a) says that $\cS(\theta)=\cS(\cU')\neq\cS(\cU)$. It remains to prove (a) without assuming (b).

Suppose that $\theta$ is on the wrong side of $H(M_i)$. If $\cU$ is on the positive side of $D(M_i)$ and $\theta$ is on the negative side, then $M_i\in\cS(\cU)$ and $M_i\notin\cS(\theta)$. So, we may assume $\cU$ is on the negative side of $D(M_i)$ and $\theta$ is on the positive side. Then $\theta(M_i)>0$. If $M_i\in\cS(\theta)$ we would be done since $M_i\notin \cS(\cU)$. So, suppose $M_i\notin\cS(\theta)$. Then, by Lemma \ref{lem: existence of admissible subobject}, there is an admissible subobject $X$ of $M_i$ so that $X\in\cS(\theta)$. But $X\notin\cS(\cU)$ since, if it were, we would get the following contradiction.

Choose $\theta_0$ a generic point on the boundary of $\cU$ in the interior of the wall $D(M_i)$. Take a sequence of points $\theta_j\in\cU$ converging to $\theta_0$. Then $\theta_j(X)>0$ for all $j$. So, $\theta_0(X)\ge0$. But $\theta_0(M_i/X)\ge0$ by definition of $D(M_i)$. So, $\theta_0(M_i/X)=0$. This contradicts our assumption that $\theta_0\in int\,D(M_i)$. Therefore, $X\notin\cS(\cU)$. So, $\cS(\cU)\neq \cS(\theta)=\cS(\cU')$.
\end{proof}
}

\section{Ghosts}\label{sec 3: ghosts}

We give definitions of three types of ghosts: ``subobject ghosts'', ``quotient ghosts'' and the new ``extension ghosts'' and give examples of all three. For example, Figures \ref{Figure11} and \ref{Figure10} show how one example can have all three kinds of ghosts. We study ``bifurcations'' of subobject ghosts and give a list of the analogous bifurcations of quotient ghosts. We will explore these in the next paper which will be about applications of ghosts to algebraic K-theory as we started to do in \cite{GrInvRedo}. We also plan to explain how extension ghosts can be used to visualize the calculation of higher Reidemeister torsion carried out in \cite{IgKlein}.

The basic setup is that we have a short exact sequence of bricks
\[
	A\to B\to C
\]
where $\Hom(A,C)=0$. (The condition $\Hom(C,A)=0$ is automatically satisfied.) There are three cases.
\begin{enumerate}
\item $B,C\in\cL$ but $A\notin\cL$. In that case we write $A=Z$ and we call the sequence $Z\to B\to C$ the \emph{ghost of $Z$ as a subobject} of $B$ and we denote this by $\cG h(Z;B)$.
\item $A,B\in\cL$ but $C\notin\cL$. Then let $Z^\ast=C$ and call this the \emph{ghost of $Z^\ast$ as a quotient object} of $B$. The notation is $\cG h^\ast(Z^\ast;B)$.
\item $A,B,C\in\cL$. In this case we get the \emph{ghost of $B$ as an unstable extension} of $A$ by $C$. We use the notation $\widetilde B=\cG h(A\to B\to C)$. (See Figures \ref{fig: Kroneker ghosts}, \ref{Figure11}, \ref{Figure12}.)
\end{enumerate}

\subsection{Subobject ghosts in torsion classes}\label{ss: subsection on subobject ghosts}

We assume that $\cL=\cG$ is a torsion class and we consider subobject ghosts $Z\to B\to C$ as above and denote it $\cG h(Z;B)$.
The simplest case is the following.

\begin{defn}
We say that $\cG h(Z;B)$ is a \emph{minimal ghost} if $B$ has no subobjects in $\cL$. Equivalently, $B$, and therefore also $C$, has no admissible quotients. They are minimal.
\end{defn}

When $\cG h(Z;B)$ is minimal, $D(B)$ and $D(C)$ are full hyperplanes and the domain of the ghost $\cG h(Z;B)$ which we denote by $D(Z;B)$ is defined to be the half-hyperplane given by
\[
	D(Z;B)=\{\theta\in\RR^n\,|\, \theta(Z)=0 \text{ and } \theta(B)\ge0\}.
\]
Thus the domain of the minimal ghost is maximal: It is as large as possible since $\theta(B)\ge0$ is always one of the boundary conditions. 

\begin{rem}\label{rem: walls get smaller}
When there are more boundary conditions, the domain gets smaller and the ghost gets bigger. ($B$ and $C$ become larger module.) Another example of this is if we embed $\cL$ in a larger set of modules, say $\cL'$. Then, for any $M\in\cL$, $D_{\cL'}(M)$ will be a subset of $D_\cL(M)$ since $D_{\cL'}(M)$ is given by $\theta(M)=0$ and $\theta(M')\ge0$ for all quotients $M'$ of $M$ in $\cL'$. So, there are more conditions giving a smaller set.
\end{rem}

In the stereographic projection of the intersection of these sets with the unit sphere $S^{n-1}$, $D(B),D(C)$ will be $n-2$ spheres and the minimal ghost $D(Z;B)$ will be the $n-2$ disk with boundary the $n-3$ sphere $\partial D(Z;B)=D(B)\cap D(C)$. See Figure \ref{fig: Case 0, minimal ghost} for an example which has three minimal subobject ghosts.

\begin{figure}[htbp]
\begin{center}
\begin{tikzpicture}
\begin{scope}
\draw[thick] (-.95,1)circle[radius=1.5cm];
\draw[thick] (.95,1)circle[radius=1.5cm];
\draw[thick] (0,-.7)circle[radius=1.5cm];
\draw(-2.7,2.2)node{$D(P_3)$};
\draw(2.7,2.2)node{$D(S_3)$};
\draw(0,-1.8)node{$D(I_2)$};
\draw[very thick,red!80!black] (-1.45,-.4)--(.5,.7);
\draw[red!80!black] (-1.5,-.6) node[left]{$Z_a$};
\draw[green!80!black] (0,-.3) node[below]{$Z_c$};
\draw[blue] (1.5,-.6) node[right]{$Z_b$};
\draw[very thick,blue] (1.45,-.4)--(-.5,.7);
\draw[very thick,green!90!black] (0,-.15)--(0,2.15);
\end{scope}
\begin{scope}[xshift=4.4cm,yshift=.5cm]
\draw[red!80!black] (0,0) node{$S_1$}; 
\draw[green!70!black] (0.5,0.6) node{$P_2$}; 
\draw (1.5,0.6) node{$I_2$};
\draw (1,1.2) node{$P_3$};
\draw[blue] (1,0) node{$S_2$}; 
\draw (2,0) node{$S_3$}; 
\draw[->] (0.13,0.15)--(.33,.4);
\draw[->] (1.13,0.15)--(1.33,0.4);
\draw[->] (0.63,0.75)--(.83,1);
\draw[->] (0.63,0.37)--(.83,.14);
\draw[->] (1.63,.37)--(1.83,.14);
\draw[->] (1.13,.97)--(1.33,.74);
\end{scope}
\begin{scope}[xshift=4cm,yshift=-.5cm]
\draw[red!80!black] (0,0)node[right]{$Z_a=\cG h(S_1,P_3)$};
\draw[blue] (0,-.6)node[right]{$Z_b=\cG h(S_2,I_2)$};
\draw[green!70!black] (0,-1.2)node[right]{$Z_c=\cG h(P_2,P_3)$};
\end{scope}
\end{tikzpicture}
\caption{Here the quiver is $1\ot 2\ot 3$. The AR quiver is indicated on the right. The objects $P_3,I_2,S_3$ are all minimal. So they give three circles in the picture. Therefore all three ghosts are minimal. For example, the domain of $\color{red!80!black}Z_a=\cG h(S_1,P_3)$ connects the intersection points of circles $D(P_3)$ and $D(I_2)$.}
\label{fig: Case 0, minimal ghost}
\end{center}
\end{figure}
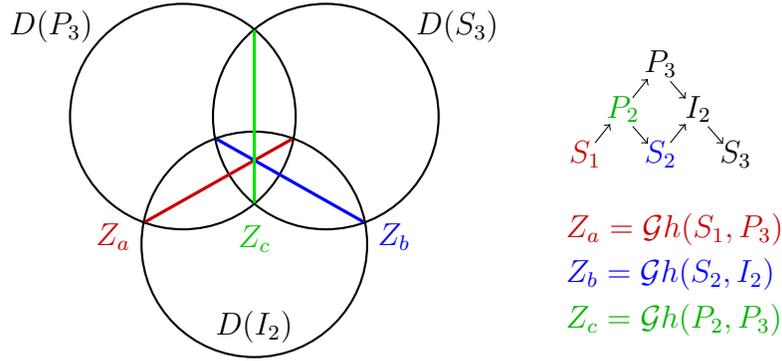

{%
\begin{figure}[htbp]
\begin{center}

\begin{tikzpicture}[scale=1.3] 
%
{
\begin{scope}
\begin{scope}
\clip rectangle (2,1.3) rectangle (0,-1.3);
\draw[thick] (0,0) ellipse [x radius=1.5cm,y radius=1.21cm];
\end{scope}
\begin{scope}[xshift=.87mm]
		\draw[thick] (1.2,.6) node[left]{\small$P_3$};
\end{scope}
\begin{scope}[xshift=-.7cm]
	\draw[thick] (0,0) circle[radius=1.4cm];
		\draw (-1,1.1) node[left]{\small$S_1$};
\end{scope}
\begin{scope}[xshift=.7cm]
	\draw[thick] (0,0) circle[radius=1.4cm];
		\draw (1,1.1) node[right]{\small$I_2$};
\end{scope}
\begin{scope}[yshift=-1.35cm]
	\draw[thick] (0,0) circle[radius=1.4cm];
	\draw (1.15,-.8) node[right]{\small$S_3$};
\end{scope}
%
\coordinate(B1) at (1.39,-1.22);
\coordinate(B2) at (-.69,-.14);
\coordinate(B) at (-.1,-.9);
\coordinate(A) at (1.24,-.69);
\coordinate(C) at (.48,-.75);
\coordinate(AC) at (.84,-.73);
\draw[red!80!black] (AC) node[above]{\tiny$Z_a$};
\draw[blue] (B) node[above]{\small$Z_b$};
\draw[thick,blue] (B1)--(B2);
\draw[thick,red!80!black] (A)--(C);
\end{scope} 
\begin{scope}[xshift=4cm,yshift=-.5cm]
\draw (0,0) node{$S_1$}; 
\draw(0,0)circle[radius=2.4mm];
\draw[red!80!black] (0.5,0.6) node{$P_2$}; 
\draw (1.5,0.6) node{$I_2$};
\draw(1.5,0.6)circle[radius=2.4mm];
\draw (1,1.2) node{$P_3$};
\draw(1,1.2)circle[radius=2.4mm];
\draw[blue] (1,0) node{$S_2$}; 
\draw (2,0) node{$S_3$}; 
\draw(2,0)circle[radius=2.4mm];
\draw[->] (0.13,0.15)--(.33,.4);
\draw[->] (1.13,0.15)--(1.33,0.4);
\draw[->] (0.63,0.75)--(.83,1);
\draw[->] (0.63,0.37)--(.83,.14);
\draw[->] (1.63,.37)--(1.83,.14);
\draw[->] (1.13,.97)--(1.33,.74);
\end{scope}

\begin{scope}[xshift=4cm,yshift=-1.3cm]
\draw[red!80!black] (0,0)node[right]{$Z_a=\cG h(P_2,P_3)$};
\draw[blue] (0,-.4)node[right]{$Z_b=\cG h(S_2,I_2)$};
\end{scope}
}
\end{tikzpicture}

\caption{This is Figure \ref{Figure00} with subobject ghosts $Z_a$ and $Z_b$ added. We also drop the $D(-)$ from the notation to save space. This example is Case (3) in Definition \ref{def: subobject ghost domains} and in Proposition \ref{prop: bifurcation list} and is further discussed in subsection \ref{sss: Case 3}.}
\label{Figure05}
\end{center}
\end{figure}
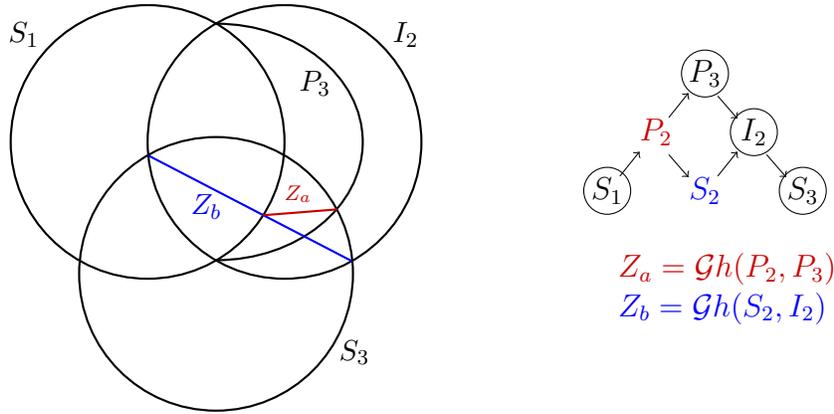
}

{
In Figure \ref{Figure05} we added ghosts to Figure \ref{Figure00}. In this example, $Z_b$ is the minimal ghost with endpoints the intersection points of the circles $D(S_3)$ and $D(I_2)$. This example also show a ``bifurcation''. When the domain $D(S_2;I_2)$ of the minimal ghost $Z_b=\cG h(S_2;I_2)$ crosses the ``splitting wall'' $D(X)=D(S_1)$, it produces another ghost $Z_a=\cG h(P_2,P_3)$.

In general, the domain of the new ghost $\cG h(Z';B')$ has an added condition coming from the splitting wall $D(X)$:
\[
	D(Z';B')=\{\theta\,|\, \theta(Z')=0, \theta(B')\ge0, \theta(X)\le0\}.
\]
Thus the boundary of a ghost domain which is not minimal has two parts: the \emph{standard part} $D(B')\cap D(C')$ and the splitting wall $D(X)$ intersected with another ghost domain.
}

We list all the possibilities and give examples. We also give a short explanation why this is the complete list in the case of a torsion class in a string algebra which includes the case of subobject ghosts in type $A_n$. The statement is that the boundary of the subobject ghost domain is a union of parts given by these boundary conditions. 

\begin{defn}\label{def: subobject ghost domains}
The domain $D(Z;B)$ of the subobject ghost of $Z$ as the kernel of $B\onto C$ is the set of all $\theta\in\RR^n$ so that $\theta(Z)=0$ and $\theta$ satisfies the following five inequalities.
\begin{enumerate}
\item[(0)] $\theta(B)\ge0$.
\item $\theta(Y)\ge0$ for all $Y$ admissible quotient of both $B$ and $C$. See Figure \ref{Figure01}. 
\item $\theta(X)\le 0$ for all $X\in\cG$ which is a submodule of $B$ disjoint from $Z$ (so that it maps to a submodule of $C$). See Figure  \ref{Figure03}. 
\item $\theta(X)\le0$ if $X\in\cG$ is a subobject of $Z$. This is the case shown in Figure \ref{Figure05}. 
\item $\theta(X)\le 0$ when $X\in\cG$ is a subobject of $C$ so that the epimorphism $B\to C/X$ is not admissible. See Figure \ref{Figure07}. 
\item $\theta(Y)\ge0$ if $f:B\to Y$ is an epimorphism with $Z+\ker f=B$ so that $\ker f$ maps onto $C$. See Figure \ref{Figure04}. 
\end{enumerate}
\end{defn}

\begin{rem}\label{rem: why list of bifurcations is complete}
    We claim that, in the case of a string algebra, this is the complete list of possible bifurcations. Since $A,B,C$ are indecomposable, $A$ and $C$ form the ends of the string for $B$, say $A$ is on the left and $C$ is the right. A bifurcation occurs when the pair $(B,C)$ crosses a wall that changes one or both of these. This happens when a common right tail of both $B,C$ is chopped off. The common tail could be a common quotient $Y$ or a common submodule $X$ and what is left is $(B',C')$. This is Cases (1) and (2). Recall that, under an extension, the smaller modules have walls that continue past the bifurcation point and the extension is a half-wall that originates at the crossing. So, the ghost for $(B,C)$ originates at this crossing and the walls for $X,Y$ and the other ghost extend through the crossing.

    When the longer ``leg'' $B$ is cut, we get $B'$ which is either a subobject or quotient object of $B$. This gives Cases (5) and (3). Finally, we can cut $C$ and get a smaller $C'$ which must be quotient of $C$. This is Case (4). There is one more possibility which is not listed: We could cut $B$ so that it becomes shorter than $C$ and we get an epimorphism $C\to B'$. This is a possible bifurcation which we will need to consider in the next paper. However, this final possible bifurcation is not listed since it cannot happen in the case of a torsion class.
\end{rem}

The analogue of Theorem \ref{thm: stable implies path crosses D(M)} for ghost walls will hold with the same proof.

\begin{defn}\label{def: stability of ghosts}
Let $\gamma_t$ be a generic smooth green path (Definition \ref{def: linear MGS}). A subobject ghost $\cG h(Z;B)$ is define to be \emph{$\gamma_t$ stable} if the following conditions hold.
\begin{enumerate}
\item[(0)] $t_C<t_B$.
\item $t_Y<t_Z$ whenever $Y$ is an admissible quotient of both $B$ and $C$.
\item $t_X>t_Z$ whenever $X\in\cG$ is a submodule of $B$ disjoint from $Z$.
\item $t_X>t_Z$ whenever $X\in \cG$ is a subobject of $Z$.
\item $t_X>t_Z$ whenever $X\in\cG$ is a subobject of $C$ so that the epimorphism $B\to C/X$ is not admissible.
\item $t_Y<t_Z$ if $f:B\to Y$ is an epimorphism whose kernel maps onto $C$.
\end{enumerate}
In cases (2),(3),(4) we say that $D(X)$ is a \emph{subobject splitting wall} for $\cG h(Z;B)$. In cases (1) and (5) we say that $D(Y)$ is a \emph{quotient objects splitting wall} for $\cG h(Z;B)$.
\end{defn}

\begin{thm}\label{thm: ghosts are stable iff path goes though domain}
A generic smooth green path $\gamma_t$ passes though the domain $D(Z;B)$ of a subobject ghost $\cG h(Z;B)$ if and only if this ghost is $\gamma_t$-stable.
\end{thm}

\begin{proof} Since $\gamma_t(X)$ is monotonically increasing we have that $t_Y<t_Z$ if and only if $\gamma_{t_Z}(Y)>0$. Similarly $\gamma_{t_Z}(X)<0$ if and only if $t_X>t_Z$. Thus $\gamma_t$ satisfies conditions (1)-(5) of Definition \ref{def: stability of ghosts} if and only if $\gamma_{t_Z}$ satisfies (1)-(5) of Definition \ref{def: subobject ghost domains}. For condition (0) we note that since $B$ is an extension of $Z$ by $C$, either $t_Z<t_B<t_C$ or $t_C<t_B<t_Z$. The second condition is equivalent to $\gamma_{t_Z}(B)>0$. Thus all conditions are equivalent and $\cG h(Z;B)$ is \emph{$\gamma_t$ stable} if and only if $\gamma_{t_Z}\in D(Z;B)$.
\end{proof}

\begin{cor}\label{cor: stability of minimal ghosts}
A minimal ghost $\cG h(Z;B)$ is $\gamma_t$ stable if and only if $t_C<t_B$. Equivalently, $\gamma_{t_B}\in D_\cL(B)$ where $\cL=add\,(\cG\cup Z)$.\qed
\end{cor}

\begin{cor}\label{cor: stability of nonminimal ghosts}
A subobject ghost $\cG h(Z;B)$ is $\gamma_t$ stable if and only if $t_C<t_B$, $t_Y<t_Z$ for all quotient object splitting walls $D(Y)$ and $t_X>t_Z$ for all subobject splitting walls $D(X)$.\qed
\end{cor}


{
It is clear that, in general, $D(Z;B)$ is a closed convex subset of the hyperplane $H(Z)$ whose boundary has six parts given by the conditions in Definition \ref{def: stability of ghosts}. The terminology in the following proposition is explained in Figure \ref{fig: general bifurcation}.
\begin{figure}[htbp]
\begin{center}
\begin{tikzpicture}[scale=1.5]
\begin{scope}
\draw[very thick,blue] (0,-.7)--(0,1);
\draw[blue] (0,.9) node[right]{$Z_b$};
\draw[very thick, red!80!black] (0,0)--(.8,.8);
\draw[red!80!black] (.8,.8) node[right]{$Z_a$};
\draw[thick] (-1,-.3)..controls (-.7,.1) and (.7,0.1)..(1,-.3); 
\draw (-1,0.05) node{$D(X)$};
\end{scope}
\begin{scope}[xshift=3.5cm]
\draw[very thick,blue] (0,-.7)--(0,1);
\draw[blue] (0,.9) node[right]{$Z_b$};
\draw[very thick, red!80!black] (0,0)--(.8,.8);
\draw[red!80!black] (.8,.8) node[right]{$Z_a$};
\draw[thick] (-1,.3)..controls (-.7,-.1) and (.7,-0.1)..(1,.3); 
\draw (-1,-.05) node{$D(Y)$};
\end{scope}
\end{tikzpicture}
\caption{When the ghost $Z_b$ hit the splitting wall $D(X)$ or $D(Y)$, it splits off a new ghost $Z_a$ on the negative side of the subobject splitting wall $D(X)$ or the positive side of the quotient object splitting wall $D(Y)$. The subobject ghost $Z_b$ is smaller than $Z_a$ even though its domain is larger. For example, $Z_b$ could be a minimal ghost, but $Z_a$ is never minimal.}
\label{fig: general bifurcation}
\end{center}
\end{figure}
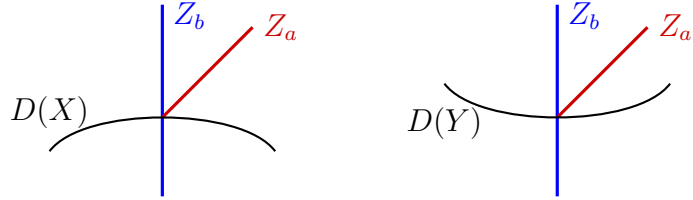
}

{
\begin{prop}\label{prop: bifurcation list}
The subobject ghosts $Z_a$ in Cases (1)-(5) of Definition \ref{def: stability of ghosts} are split off from smaller subobject ghosts $Z_b$ as follows.
\begin{enumerate}
\item When $Y$ is an admissible quotient of both $B$ and $C$, $Z_a$ splits off from $Z_b=\cG h(Z';B')$ where $Z'=Z$ and $B',C'$ are the kernels of the epimorphisms $B\onto Y$ and $C\onto Y$. The splitting wall is $D(Y)$.
\item When $X\in\cG$ is a submodule of $B$ disjoint from $Z$ so that it maps to a submodule of $C$, $Z_a$ splits off from $Z_b=\cG h(Z';B')$ where $Z'=Z$, $B'=B/X$ and $C'=C/X$. The splitting wall is $D(X)$.
\item If $X\in\cG$ is a subobject of $Z$, $Z_a$ is split off from $Z_b=\cG h(Z';B')$ where $Z'=Z/X$, $B'=B/X$ and $C'=C$. The splitting wall is $D(X)$. The exact sequence $X\to Z\to Z'$ shows that $Z'\notin\cG$ since otherwise $Z\in\cG$ contrary to assumptions.
\item When $X\in\cG$ is a subobject of $C$ so that the epimorphism $B\onto C/X$ is not admissible, then $Z_a$ is split off of $Z_b=\cG h(Z';B')$ where $Z'=\ker(B\onto C/X)$ ($Z'\notin\cG$ by assumption), $B'=B$ and $C'=C/X$. The splitting wall is $D(X)$. 
\item If $f:B\to Y$ is an epimorphism with $Z+\ker f=B$ (so that $\ker f$ maps onto $C$), then $Z_a$ has split off $Z_b=\cG h(Z';B')$ where $B'=\ker f$, $C'=C$ and $Z'=\ker(Z\to Y)$. (As in (3), the exact sequence $Z'\to Z\to Y$ shows that $Z'\notin\cG$.) The splitting wall is $D(Y)$.
\end{enumerate}
\end{prop}
{
In Proposition \ref{prop: bifurcation list} we are assuming that $B',C',Z'$ as described in this list are indecomposable. This holds in all the examples and we claim that this list is complete and accurate for subobject ghosts in torsion classes for type $A_n$ path algebras. An, in fact, only cases (1),(3),(4) occur for linearly oriented $A_n$.
}
}

Figure \ref{Figure05} shows an example of Case (3). Here, $X=S_1\in\cG$ is a subobject of $Z=P_2$ in the ghost $Z_a=\cG h(P_2;P_3)$. In the example, $B=P_3$. So, $B'=B/X=I_2$, $Z'=Z/X=S_2$ and $C'=C=S_3$ making $Z_b=\cG h(S_2; I_2)$. Examples of the other four cases are shown in subsection \ref{ss: examples} below. 

But first we describe the dual case of quotient object ghosts in torsion free classes.

{
\subsection{Quotient object ghosts}\label{ss: quotient object ghosts}

Given a class of modules $\cL$ as in Section \ref{sec1: definitions} and a brick $Z^\ast$ not in $\cL$, a \emph{quotient object ghost} of $Z^\ast$ is an isomorphism class of short exact sequences
\[
	A\to B\to Z^\ast
\]
where $A,B$ are bricks in $\cL$ and $\Hom(A,Z^\ast)=0$. In Figure \ref{fig: Kroneker ghosts}, the short exact sequence is $S_1\to M\to S_2$ where $S_1,M\in\cL$ but $S_2\notin\cL$. This sequence is a quotient ghost of $S_2$.

In a torsion-free class $\cF$, it is easy to see that a morphism $A\to B$ is {admissible} if its cokernel lies in $\cF$ since its kernel and image automatically lie in $\cF$. It is convenient to use admissible subobjects to describe the support of any object: 
\[
	D(B)=\{\theta\in\RR^n\,|\, \theta(B)=0, \theta(A)\le0\text{ for all admissible subobjects }A\subset B\}.
\]
We define a brick $B\in\cF$ to be \emph{minimal} if it has no admissible subobjects. $D(B)=H(B)$ in that case. We say that a quotient object ghost $A\to B\to Z^\ast$ is \emph{minimal} if both $A$ and $B$ are minimal objects. The domain of such a minimal quotient object ghost is
\[
	D^\ast(Z^\ast;B):=\{\theta\in \RR^n\,|\,\theta(Z^\ast)=0\text{ and }\theta(B)\le0\}.
\]
As in the case of subobject ghosts, a generic smooth green path will cross a quotient object domain $D^\ast(Z^\ast;B)$ (necessarily in its interior) if and only if the quotient object ghost is stable. Instead of defining the concept of stability and proving this statement, we will simply use this statement as the definition of when quotient object ghosts are stable. With this simplifying convention we have:

\begin{thm}\label{thm: when minimal quotient ghosts are stable}
Given a generic smooth green path $\gamma_t$, a minimal quotient ghost $A\to B\to Z^\ast$ will be stable if and only if $t_B<t_A$ in which case we will necessarily have $t_{Z^\ast}<t_B<t_A$. \qed
\end{thm}

\begin{defn}\label{def: domain of quotient ghost}
The domain interior $int\,D^\ast(Z^\ast;B)$ of a quotient ghost $Z_a^\ast=(A\to B\to Z^\ast)$ for a torsion-free class $\cF$ is the subset of the hyperplane $H(Z^\ast)$ given by the following conditions.
\begin{enumerate}
\item[(0)] $\theta(B)<0$.
\item $\theta(X)<0$ if $X\in \cF$ is an admissible subobject of both $A$ and $B$.
\item $\theta(Y)>0$ if $Y\in\cF$ is a quotient of both $A$ and $B$.
\item $\theta(Y)>0$ if $Y\in\cF$ is a quotient of $Z^\ast$.
\item $\theta(Y)>0$ if $Y\in \cF$ is a quotient of $A$ with indecomposable kernel $A'$ so that $B/A'$ is not in $\cF$.
\item $\theta(X)<0$ if $X$ is an admissible subobject of $B$ disjoint from $A$.
\end{enumerate}
In each case, $X$ will be called a \emph{splitting subobject} and $Y$ will be called a \emph{splitting quotient object}.
\end{defn}

\begin{rem}\label{rem: duality functor}
Applying the duality functor $D=\Hom(-,K)$ takes torsion classes to torsion-free class and converts subobject ghosts $Z\to B\to C$ into quotient object ghosts $DC\to DB\to DZ$. Definition \ref{def: domain of quotient ghost} is the dual of Definition \ref{def: subobject ghost domains}.
\end{rem}

The following theorem is the dual of Corollary \ref{cor: stability of nonminimal ghosts}.
\begin{thm}\label{thm: stability of quotient ghosts}
Let $\gamma_t$ be a generic smooth green path. Then the quotient ghost $A\to B\to Z^\ast$ is $\gamma_t$-stable if and only if $t_B<t_A$ and $t_Y<t_{Z^\ast}<t_X$ for all splitting subobjects $X$ and splitting quotient objects $Y$ listed in Definition \ref{def: domain of quotient ghost}.\qed
\end{thm}
}

\subsection{Examples of bifurcation}\label{ss: examples}

We give examples of Cases (1), (2), (4), (5) in Definition \ref{def: subobject ghost domains} and Proposition \ref{prop: bifurcation list} and discuss further Case (3).


{
\subsubsection{Case $(1)$}\label{sss: Case 1}

When $Y$ is an admissible quotient of both $B$ and $C$.
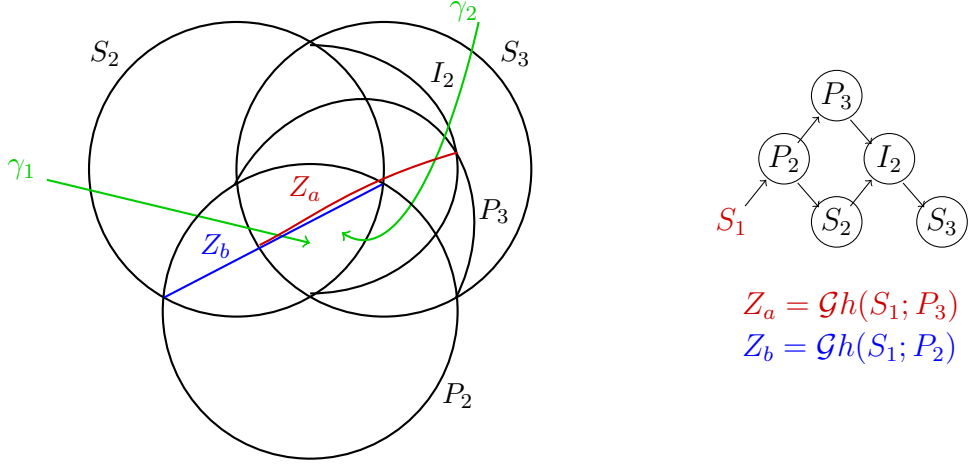
\begin{figure}[htbp]
\begin{center}

\begin{tikzpicture}[scale=1.4] 
%
{
\begin{scope}

\begin{scope}
\clip rectangle (2,1.3) rectangle (0,-1.3);
\draw[thick] (0,0) ellipse [x radius=1.4cm,y radius=1.18cm];
\end{scope}

\begin{scope}[xshift=3.3mm,yshift=-6.9mm]
\begin{scope}[rotate=63]
\clip rectangle (2,1.3) rectangle (0,-1.3);
\draw[thick] (0,0) ellipse [x radius=1.4cm,y radius=1.18cm];
\end{scope}
\end{scope}

\begin{scope}[xshift=.87mm]
		\draw[thick] (1.4,.9) node[left]{\small$I_2$}; 
		\draw[thick] (1.66,-.4) node{\small$P_3$}; 
\end{scope}
\begin{scope}[xshift=-.7cm]
	\draw[thick] (0,0) circle[radius=1.4cm];
		\draw (-1,1.1) node[left]{\small$S_2$}; 
\end{scope}
\begin{scope}[xshift=.7cm]
	\draw[thick] (0,0) circle[radius=1.4cm];
		\draw (1,1.1) node[right]{\small$S_3$};
		\draw[thick,green!80!black,->] (.9,1.4)..controls (.7,.5) and (.2,-1)..(-.4,-.6);
		\draw[green!80!black] (1,1.5) node[left]{$\gamma_2$};
\end{scope}
\draw[thick,green!80!black,->] (-2.5,-.1)--(0,-.7);
\draw[green!80!black] (-2.5,0) node[left]{$\gamma_1$};
\begin{scope}[yshift=-1.35cm]
	\draw[thick] (0,0) circle[radius=1.4cm];
	\draw (1.15,-.8) node[right]{\small$P_2$}; 
\end{scope}
%
\coordinate(B1) at (-1.39,-1.22);
\coordinate(B2) at (.69,-.14);
\coordinate(B) at (-.9,-.95);
\coordinate(A) at (1.4,.16);
\coordinate(C) at (-.48,-.72);
\coordinate(AC) at (-0.05,-.4);
\draw[red!80!black] (AC) node[above]{\small$Z_a$};
\draw[blue] (B) node[above]{\small$Z_b$};
\draw[thick,blue] (B1)--(B2);
\draw[thick,red!80!black] (A)..controls (.4,-.13) and (0,-.48)..(C); 
\end{scope} 
}
{ 
\begin{scope}[xshift=4cm,yshift=-.5cm]
\draw[red!80!black] (0,0) node{$S_1$}; 
\draw(0.5,0.6)circle[radius=2.4mm];
\draw (0.5,0.6) node{$P_2$}; 
\draw (1.5,0.6) node{$I_2$};
\draw(1.5,0.6)circle[radius=2.4mm];
\draw (1,1.2) node{$P_3$};
\draw(1,1.2)circle[radius=2.4mm];
\draw (1,0) node{$S_2$}; 
\draw(1,0)circle[radius=2.4mm];
\draw (2,0) node{$S_3$}; 
\draw(2,0)circle[radius=2.4mm];
\draw[->] (0.13,0.15)--(.33,.4);
\draw[->] (1.13,0.15)--(1.33,0.4);
\draw[->] (0.63,0.75)--(.83,1);
\draw[->] (0.63,0.37)--(.83,.14);
\draw[->] (1.63,.37)--(1.83,.14);
\draw[->] (1.13,.97)--(1.33,.74);
\end{scope}
\begin{scope}[xshift=4cm,yshift=-1.3cm]
\draw[red!80!black] (0,0)node[right]{$Z_a=\cG h(S_1;P_3)$};
\draw[blue] (0,-.4)node[right]{$Z_b=\cG h(S_1;P_2)$};
\end{scope}
}
\end{tikzpicture}

\caption{This is Case (1). We saw this in \cite{GrInvRedo}. Here the same missing module $S_1$ produces two subobject ghosts. Two green paths are indicated.}
\label{Figure01}
\end{center}
\end{figure}

This case, shown in Figure \ref{Figure01}, has two ghosts of the same missing module $Z=S_1$. This is because $B\to C$ and $B'\to C'$ form a Cartesian square. The horizontal arrows have the same kernel $Z=S_1$ and the vertical arrows have the same cokernel $Y=S_3$ which gives the splitting wall $D(S_3)$.
\[
\xymatrixrowsep{10pt}\xymatrixcolsep{10pt}
\xymatrix{
\color{blue}Z_b: &Z'\ar[d]^=\ar[r]&B'\ar[d]\ar[r] &
	C'\ar[d]& &S_1\ar[d]^=\ar[r]& P_2\ar[d]\ar[r] &
	S_2\ar[d]\\
\color{red}Z_a: &Z\ar[r] & B\ar[r]\ar[d]& 
	C\ar[d]& &S_1\ar[r] & P_3\ar[r]\ar[d]& 
	I_2\ar[d]\\
&	& Y\ar[r]^=&Y&&	 & S_3\ar[r]^=& S_3
	} 
\]
{
The green paths $\gamma_1,\gamma_2$ give the following MGSs where $(I_2)$ and $(P_3)$ indicate the positions of unstable modules, i.e, $t_{S_2}<t_{I_2}<t_{P_2}<t_{P_3}<t_{S_3}<t_{S_1}$ where $t_M$ is the time the green path crosses the hyperplane $H(M)$.
\[
	S_2,(I_2),P_2,(P_3),S_3,Z_a,Z_b  \quad \text{ and }\quad  S_3,I_2,P_3,Z_a,P_2,S_2 
\]
In the first sequence $Z_a$ is stable because the unstable modules $I_2$ and $P_3$ come in that order and $S_3$ comes before $Z_a$ (and the concurrent $Z_b$). $Z_b$ is stable because $S_2$ comes before $P_2$. In the second sequence $Z_a$ is stable because of the subsequence $S_3,I_2,P_3,Z_a$. $Z_b$ is unstable since $S_2$ comes after $P_2$.
}
}


{
\subsubsection{Case $(2)$}\label{sss: Case 2} When $X\in\cG$ is a submodule of both $B$ and $C$.
\begin{figure}[htbp]
\begin{center}

\begin{tikzpicture}[scale=1.4] 
%
{
\begin{scope}

\begin{scope}
\clip rectangle (2,1.3) rectangle (0,-1.3);
\draw[thick] (0,0) ellipse [x radius=1.4cm,y radius=1.2cm];
\end{scope}

\begin{scope}[xshift=-3.5mm,yshift=-6.9mm]
\begin{scope}[rotate=-63]
\clip rectangle (2,1.3) rectangle (0,-1.3);
\draw[thick] (0,0) ellipse [x radius=1.4cm,y radius=1.18cm];
\end{scope}
\end{scope}

\begin{scope}[xshift=.87mm]
		\draw[thick] (1.4,.9) node[left]{\small$I_1$}; 
\end{scope}
\begin{scope}[xshift=-.7cm]
	\draw[thick] (0,0) circle[radius=1.4cm];
		\draw (-1,1.1) node[left]{\small$S_1$}; 
\end{scope}
\begin{scope}[xshift=.7cm]
	\draw[thick] (0,0) circle[radius=1.4cm];
		\draw (1,1.1) node[right]{\small$S_2$}; 
\end{scope}
\begin{scope}[yshift=-1.35cm]
	\draw[thick] (0,0) circle[radius=1.4cm];
	\draw (1.15,-.8) node[right]{\small$I_3$}; 
\end{scope}
\draw (0,-2.23) node{\small$P_2$};
%
\coordinate(B1) at (1.39,-1.22);
\coordinate(B2) at (-.69,-.14);
\coordinate(B) at (0,-.5);
\coordinate(A) at (.85,-.98);
\coordinate(C) at (.48,-.78);
\coordinate(AC) at (.56,-.75);
\draw[red!80!black] (AC) node[below]{\tiny$Z_a$};
\draw[blue] (B) node[above]{\small$Z_b$};
\draw[thick,blue] (B1)--(B2);
\draw[thick,red!80!black] (A)--(C); 
\end{scope} 
}
{ 
\begin{scope}[xshift=4cm,yshift=-.5cm]
\draw (0,1.2) node{$S_1$}; 
\draw(0,1.2)circle[radius=2.4mm];
\draw(0.5,0.6)circle[radius=2.4mm];
\draw (0.5,0.6) node{$P_2$}; 
\draw (1.5,0.6) node{$S_2$}; 
\draw(1.5,0.6)circle[radius=2.4mm];
\draw (1,1.2) node{$I_3$}; 
\draw(1,1.2)circle[radius=2.4mm];
\draw (1,0) node{$I_1$}; 
\draw(1,0)circle[radius=2.4mm];
\draw[red!80!black] (0,0) node{$S_3$}; 
\draw[->] (0.13,0.15)--(.33,.4);
\draw[->] (1.13,0.15)--(1.33,0.4);
\draw[->] (0.63,0.75)--(.83,1);
\draw[->] (0.63,0.37)--(.83,.14);
\draw[->] (1.13,.97)--(1.33,.74);
\draw[->] (.13,.97)--(.33,.74);
\end{scope}

\begin{scope}[xshift=4cm,yshift=-1.3cm]
\draw[red!80!black] (0,0)node[right]{$Z_a=\cG h(S_3;P_2)$};
\draw[blue] (0,-.4)node[right]{$Z_b=\cG h(S_3,I_3)$};
\end{scope}
\draw[thick,green!90!black,->] (2.5,-.5)--(0,-.95);
\draw[green!90!black] (2.5,-.5) node[below]{$\gamma$};
}
\end{tikzpicture}

\caption{This is an example of Case (2). This is the quiver $1\ot 2\to 3$. The AR quiver is shown on the right with torsion class $\,^\perp S_3$ circled. The minimal ghost $Z_b$ splits off a large ghost (with smaller domain) $Z_a$ when it crosses the splitting wall $D(S_1)$. The domain of $Z_a$ is very short, going from $D(S_1)$ to its source $D(I_1)\cap D(P_2)$. Since these are ghosts of the same missing module, their domains are colinear. However, $Z_a$ is slightly underneith $Z_b$ because there is a homomorphism $Z_a\to Z_b$. One green path is shown. }
\label{Figure03} 
\end{center}
\end{figure}
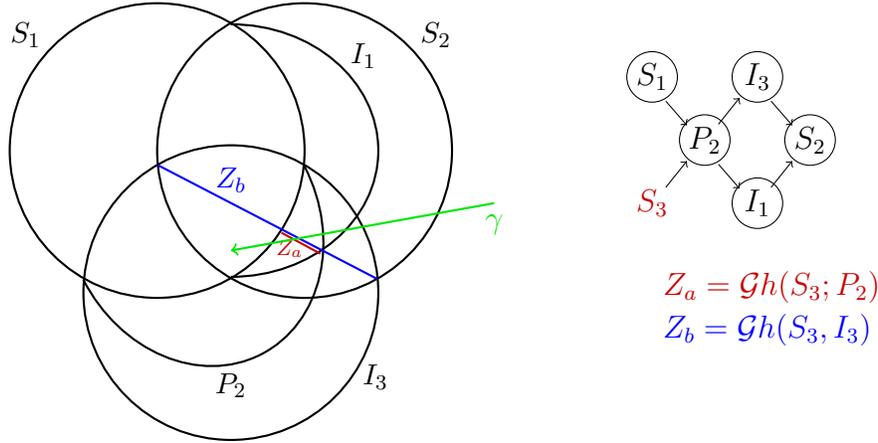
This case, shown in Figure \ref{Figure03}, is similar to Case (1) in that we have two ghosts of the same missing module. These form a Cartesian square and we interpret this as a short exact sequence $S_1\to Z_a\to Z_b$. The splitting wall is $D(S_1)$.
\[
\xymatrixrowsep{10pt}\xymatrixcolsep{10pt}
\xymatrix{
\color{blue}Z_b: &Z'\ar[r]& B'\ar[r] &
	C'& &S_3\ar[r]& I_3\ar[r] &
	S_2\\
\color{red}Z_a: &Z\ar[u]_=\ar[r] & B\ar[r]\ar[u]& 
	C\ar[u]& &S_3\ar[u]_=\ar[r] & P_2\ar[r]\ar[u]& 
	I_1\ar[u]\\
&	& X\ar[r]^=\ar[u]& X\ar[u]&&	& S_1\ar[r]^=\ar[u]& S_1\ar[u]
	} 
\]

By a \emph{maximal green sequence with ghosts} we mean the walls crossed by sequence of wall crossings (including the ghost walls) in the positive direction. It is clear that, if we delete the ghost walls, we will be left with a standard MGS. The wall crossing sequence of any generic smooth green path will be a MGS with ghosts.

\begin{eg}\label{eg: MGS with ghosts 2}
One can see that the only MGSs with ghosts which include the subobject ghost $Z_a=\cG h(S_3;P_2)$ are the following where $Z_b=\cG h(S_3;I_3)$.
\[
	S_2,I_3,I_1,P_2,Z_b,Z_a,S_1\quad \text{and}\quad S_2,I_1,I_3,P_2,Z_b,Z_a,S_1.
\]
The first sequence is given by the green path $\gamma$ in Figure \ref{Figure03}. These illustrate Theorem \ref{thm: ghosts are stable iff path goes though domain} as follows. The ghost $Z_b$ is stable because of the subsequence $S_2,I_3,Z_b$ in both cases. This is condition (0): $C'$ before $B'$. $Z_a$ is stable because we have the subsequence $I_1,P_2,Z_a,S_1$. In general we must have: $C,B,Z_a,X$. We put $Z_b$ before $Z_a$ because of the exact sequence $X\to Z_a\to Z_b$. However, this is purely heuristic since the domains of $Z_a,Z_b$ are on top of each other and any path $\gamma_t$ will cross both at the same time.

If we go slightly above and below these two paths, $Z_a$ becomes unstable and we get MGSs: 
\[
	S_2,I_1,I_3,P_2,S_1,Z_b\quad \text{and}\quad S_2,I_3,Z_b,P_2,I_1,S_1
\]
In the first sequence, $Z_a$ is unstable because $S_1$ came before $Z_b$ (and thus also before $Z_a$ since they are concurrent). $Z_b$ is still stable in both cases because of the subsequence $S_2,I_3,Z_b$. In the second sequence, $Z_a$ is unstable since $P_2$ comes before $I_1$.
\end{eg}

}


\subsubsection{Case $(3)$}\label{sss: Case 3} When $X\in\cG$ is a subobject of $Z$. In that case, $X$ is also a subobject of $B$ and we get the following diagram. See also Figure \ref{Figure05} reproduced on the right.
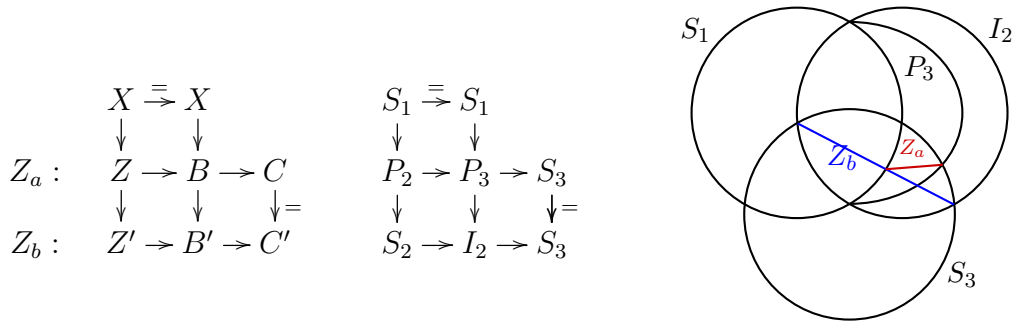
\vs5

\begin{minipage}{0.6\textwidth}
{ 
\[
\xymatrixrowsep{10pt}\xymatrixcolsep{10pt}
\xymatrix{
&X \ar[d]\ar[r]^=&X \ar[d]& &&S_1\ar[d]\ar[r]^=&S_1\ar[d]& \\
Z_a: &Z\ar[d]\ar[r]& B\ar[d]\ar[r] &
	C\ar[d]^= & &P_2\ar[d]\ar[r]& P_3\ar[d]\ar[r] &
	S_3\ar[d]\ar[d]^=\\
Z_b: &Z'\ar[r] & B'\ar[r]& 
	C'& &S_2\ar[r] & I_2\ar[r]& 
	S_3
	} 
\]
} 
\end{minipage}
\hfill 
\begin{minipage}{0.35\textwidth}
%
{
\begin{tikzpicture}
%
{
\begin{scope}
\begin{scope}
\clip rectangle (2,1.3) rectangle (0,-1.3);
\draw[thick] (0,0) ellipse [x radius=1.5cm,y radius=1.21cm];
\end{scope}
\begin{scope}[xshift=.87mm]
		\draw[thick] (1.2,.6) node[left]{\small$P_3$};
\end{scope}
\begin{scope}[xshift=-.7cm]
	\draw[thick] (0,0) circle[radius=1.4cm];
		\draw (-1,1.1) node[left]{\small$S_1$};
\end{scope}
\begin{scope}[xshift=.7cm]
	\draw[thick] (0,0) circle[radius=1.4cm];
		\draw (1,1.1) node[right]{\small$I_2$};
\end{scope}
\begin{scope}[yshift=-1.35cm]
	\draw[thick] (0,0) circle[radius=1.4cm];
	\draw (1.15,-.8) node[right]{\small$S_3$};
\end{scope}
%
\coordinate(B1) at (1.39,-1.22);
\coordinate(B2) at (-.69,-.14);
\coordinate(B) at (-.1,-.9);
\coordinate(A) at (1.24,-.69);
\coordinate(C) at (.48,-.75);
\coordinate(AC) at (.84,-.73);
\draw[red!80!black] (AC) node[above]{\tiny$Z_a$};
\draw[blue] (B) node[above]{\small$Z_b$};
\draw[thick,blue] (B1)--(B2);
\draw[thick,red!80!black] (A)--(C);
\end{scope} 

}
\end{tikzpicture}
}
\end{minipage}%
\vs2

{
The ghosts $Z_a,Z_b$ are oriented upward. Subobject ghosts are always oriented in the same direction as $B$ which is the larger of the two modules $B,C$ which support it. In this case $B=I_2$ and $B'=P_3$. For example, we have MGSs which pass through these from below:
\[
	S_1,S_3,I_2,Z_b\text{   and   } I_2,S_3,P_3,Z_a,S_1
\]
In the first case $Z_b$ is stable because $S_3$ comes before $I_2$ and $Z_a$ is unstable since $S_1$ comes too soon. In the second MGS, $Z_a$ is stable since $S_3$ comes before $P_3$ and $Z_a$ comes before the subobject splitting module $S_1$. $Z_b$ is unstable since $I_2$ comes before $S_3$.
}

{
\begin{eg}\label{eg: example of duality}
We give an example of the duality functor $D$. This functor takes bricks to bricks and short exact sequence to short exact sequences. If reverses all arrows and it reverses all MGSs. Thus we obtain quotient ghosts $Z_a^\ast$ and $Z_b^\ast$ given by the rows of the diagram:
\[
\xymatrixrowsep{10pt}\xymatrixcolsep{10pt}
\xymatrix{
\color{blue}Z_b^\ast: &A'\ar[r]\ar[d]^=& B'\ar[r]\ar[d] &
	{Z^\ast}'\ar[d]& &DS_3\ar[r]\ar[d]^=&D I_2\ar[r]\ar[d] &
	DS_2\ar[d]\\
\color{red}Z_a^\ast: &A\ar[r] & B\ar[r]\ar[d]& 
	Z^\ast\ar[d]& &DS_3\ar[r] & DP_3\ar[r]\ar[d]& 
	DP_2\ar[d]\\
&	& Y\ar[r]^=& Y&&	& DS_1\ar[r]^=& DS_1
	} 
\]
The MGSs are:
\[
	Z_b^\ast,DI_2,DS_3,DS_1\text{  and  } DS_1,Z_a^\ast,DP_3,DS_3,DI_2.
\]
For example, the cokernel of any nonzero morphism $DS_3\to DI_2$ is not in $\cF$ since the kernel of any nonzero morphism $I_2\to S_3$ is not in $\cG$.
\end{eg}
}

{
\subsubsection{Case $(4)$}\label{sss: Case 4}  When $X\in\cG$ is a subobject of $C$ so that the epimorphism $B\onto C/X$ is not admissible.
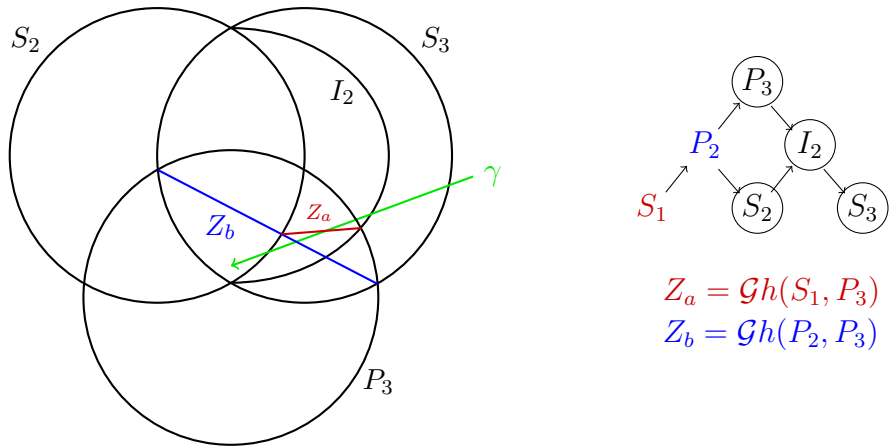
\begin{figure}[htbp]
\begin{center}

\begin{tikzpicture}[scale=1.4] 
%
{
\begin{scope}
\begin{scope}
\clip rectangle (2,1.3) rectangle (0,-1.3);
\draw[thick] (0,0) ellipse [x radius=1.5cm,y radius=1.21cm];
\end{scope}
\begin{scope}[xshift=.87mm]
		\draw[thick] (1.2,.6) node[left]{\small$I_2$}; 
\end{scope}
\begin{scope}[xshift=-.7cm]
	\draw[thick] (0,0) circle[radius=1.4cm];
		\draw (-1,1.1) node[left]{\small$S_2$}; 
\end{scope}
\begin{scope}[xshift=.7cm]
	\draw[thick] (0,0) circle[radius=1.4cm];
		\draw (1,1.1) node[right]{\small$S_3$};
\end{scope}
\draw[green!90!black,->,thick] (2.3,-.2)--(0,-1.05);
\draw[green!90!black] (2.3,-.2) node[right]{$\gamma$};
\begin{scope}[yshift=-1.35cm]
	\draw[thick] (0,0) circle[radius=1.4cm];
	\draw (1.15,-.8) node[right]{\small$P_3$}; 
\end{scope}
%
\coordinate(B1) at (1.39,-1.22);
\coordinate(B2) at (-.69,-.14);
\coordinate(B) at (-.1,-.9);
\coordinate(A) at (1.24,-.69);
\coordinate(C) at (.48,-.75);
\coordinate(AC) at (.84,-.73);
\draw[red!80!black] (AC) node[above]{\tiny$Z_a$};
\draw[blue] (B) node[above]{\small$Z_b$};
\draw[thick,blue] (B1)--(B2);
\draw[thick,red!80!black] (A)--(C);
\end{scope} 
}

{ 
\begin{scope}[xshift=4cm,yshift=-.5cm]
\draw[red!80!black] (0,0) node{$S_1$}; 
\draw(1,0)circle[radius=2.4mm];
\draw[blue] (0.5,0.6) node{$P_2$}; 
\draw (1.5,0.6) node{$I_2$};
\draw(1.5,0.6)circle[radius=2.4mm];
\draw (1,1.2) node{$P_3$};
\draw(1,1.2)circle[radius=2.4mm];
\draw (1,0) node{$S_2$}; 
\draw (2,0) node{$S_3$}; 
\draw(2,0)circle[radius=2.4mm];
\draw[->] (0.13,0.15)--(.33,.4);
\draw[->] (1.13,0.15)--(1.33,0.4);
\draw[->] (0.63,0.75)--(.83,1);
\draw[->] (0.63,0.37)--(.83,.14);
\draw[->] (1.63,.37)--(1.83,.14);
\draw[->] (1.13,.97)--(1.33,.74);
\end{scope}

\begin{scope}[xshift=4cm,yshift=-1.3cm]
\draw[red!80!black] (0,0)node[right]{$Z_a=\cG h(S_1,P_3)$};
\draw[blue] (0,-.4)node[right]{$Z_b=\cG h(P_2,P_3)$};
\end{scope}
}

\end{tikzpicture}

\caption{This is an example of Case (4). $X=S_2$ maps to $C=I_2$ to give $C'=S_3$. Thus $D(S_2)$ splits $Z_a$ off of $Z_b$. The quiver is $1\ot 2\ot3$ with AR-quiver shown on the right with torsion class circled.  The unique green path through the domain of $Z_a$ is indicated.}
\label{Figure07}
\end{center}
\end{figure}

Figure \ref{Figure07} shows an example of Case (4). The ghosts are related by the following diagram. In Case (4) we always have this, an example of the Snake Lemma, which says that $X=\ker(C\to C')=\coker(Z\to Z')$ when $B=B'$.
\[
\xymatrixrowsep{10pt}\xymatrixcolsep{10pt}
\xymatrix{
&&& X\ar[d]&&&& S_2\ar[d]\\
Z_a: &Z\ar[d]\ar[r]& B\ar[d]^=\ar[r] &
	C\ar[d] & &S_1\ar[d]\ar[r]& P_3\ar[d]^=\ar[r] &
	I_2\ar[d]\\
Z_b: &Z'\ar[r]\ar[d] & B'\ar[r]& 
	C'& &P_2\ar[r]\ar[d] & P_3\ar[r]& 
	S_3\\
&	X &&&&	S_2
	} 
\]

{
The unique MGS containing $Z_a$ given by the green path in Figure \ref{Figure07} is
\[
	S_3,I_2,P_3,Z_a,Z_b,S_2
\]
$Z_b$ is stable since $C'=S_3$ comes before $B'=P_3$. $Z_a$ is stable because $C=I_2$ comes before $B=P_3$ and $Z_a$ comes before the splitting subobject $X=S_2$.

This example has the feature that $Z_a$ comes before $Z_b$ in the MGS even though there is a morphism $Z_a\to Z_b$. This is unexpected, however, this does not contradict any of the definitions or theorems.
}

}


{
\subsubsection{Case $(5)$}\label{sss: Case 5} When $f:B\to Y$ is an epimorphism so that $\ker f$ maps onto $C$.
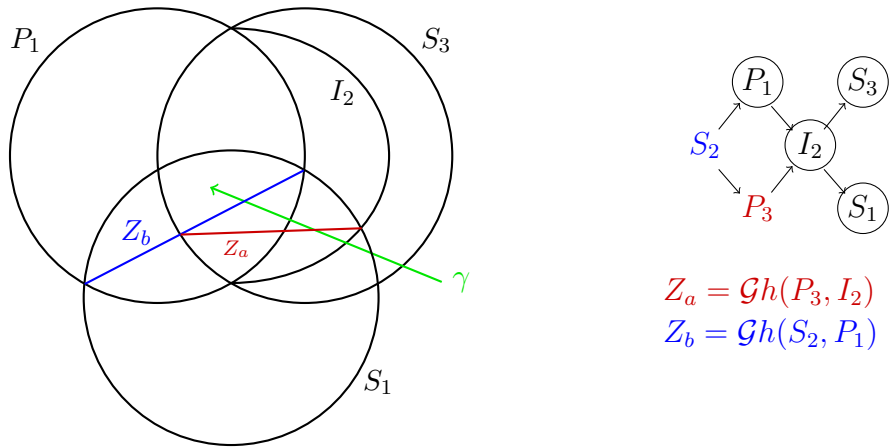
\begin{figure}[htbp]
\begin{center}

\begin{tikzpicture}[scale=1.4] 
{
\begin{scope}
\begin{scope}
\clip rectangle (2,1.3) rectangle (0,-1.3);
\draw[thick] (0,0) ellipse [x radius=1.5cm,y radius=1.21cm];
\end{scope}
\begin{scope}[xshift=.87mm]
		\draw[thick] (1.2,.6) node[left]{\small$I_2$}; 
\end{scope}
\begin{scope}[xshift=-.7cm]
	\draw[thick] (0,0) circle[radius=1.4cm];
		\draw (-1,1.1) node[left]{\small$P_1$}; 
\end{scope}
\begin{scope}[xshift=.7cm]
	\draw[thick] (0,0) circle[radius=1.4cm];
		\draw (1,1.1) node[right]{\small$S_3$}; 
\end{scope}
\draw[thick,green!90!black,->] (2,-1.2)--(-.2,-.3);
\draw[green!90!black] (2,-1.2) node[right]{$\gamma$};
\begin{scope}[yshift=-1.35cm]
	\draw[thick] (0,0) circle[radius=1.4cm];
	\draw (1.15,-.8) node[right]{\small$S_1$}; 
\end{scope}
%
\coordinate(B1) at (-1.39,-1.22);
\coordinate(B2) at (.69,-.14);
\coordinate(B) at (-.9,-.95);
\coordinate(A) at (1.24,-.69);
\coordinate(C) at (-.48,-.75);
\coordinate(AC) at (0.05,-1.07);
\draw[red!80!black] (AC) node[above]{\tiny$Z_a$};
\draw[blue] (B) node[above]{\small$Z_b$};
\draw[thick,blue] (B1)--(B2);
\draw[thick,red!80!black] (A)--(C);
\end{scope} 
}
{ 
\begin{scope}[xshift=4cm,yshift=-.5cm]
\draw (2,1.2) node{$S_3$}; 
\draw(2,1.2) circle[radius=2.4mm];
\draw[blue] (0.5,0.6) node{$S_2$}; 
\draw (1.5,0.6) node{$I_2$};
\draw(1.5,0.6)circle[radius=2.4mm];
\draw (1,1.2) node{$P_1$}; 
\draw(1,1.2)circle[radius=2.4mm];
\draw[red!80!black] (1,0) node{$P_3$}; 
\draw (2,0) node{$S_1$}; 
\draw(2,0)circle[radius=2.4mm];
\draw[->] (1.13,0.15)--(1.33,0.4);
\draw[->] (0.63,0.75)--(.83,1);
\draw[->] (0.63,0.37)--(.83,.14);
\draw[->] (1.63,.37)--(1.83,.14);
\draw[->] (1.13,.97)--(1.33,.74);
\draw[->] (1.63,0.75)--(1.83,1);
\end{scope}

\begin{scope}[xshift=4cm,yshift=-1.3cm]
\draw[red!80!black] (0,0)node[right]{$Z_a=\cG h(P_3,I_2)$};
\draw[blue] (0,-.4)node[right]{$Z_b=\cG h(S_2,P_1)$};
\end{scope}
}
\end{tikzpicture}

\caption{This shows Case (5). The quiver is $1\to 2\ot 3$ with AR-quiver displayed on the right with torsion class $\,^\perp P_3$ circled. Here $C=C'$, but $B=I_2$ and $Z=P_3$ map onto $Y=S_3$. The minimal ghost $Z_b$ gives birth to $Z_a$ when it crosses the splitting wall $D(S_3)$. One green path is drawn.}
\label{Figure04}
\end{center}
\end{figure}

An example of Case (5) is shown in Figure \ref{Figure04}. The ghosts are related by the following diagram which behaves like a short exact sequence $Z_b\to Z_a\to S_3$.
\[
\xymatrixrowsep{10pt}\xymatrixcolsep{10pt}
\xymatrix{
\color{blue}Z_b: &Z'\ar[d]\ar[r]& B'\ar[d]\ar[r] &
	C'\ar[d]^=& &S_2\ar[d]\ar[r]& P_1\ar[d]\ar[r] &
	S_1\ar[d]^=\\
\color{red}Z_a: &Z\ar[r]\ar[d] & B\ar[r]\ar[d]& 
	C& &P_3\ar[r]\ar[d] & I_2\ar[r]\ar[d]& 
	S_1\\
&	Y\ar[r]^= & Y&&&	S_3\ar[r]^= & S_3
	} 
\]
{
The green path in Figure \ref{Figure04} gives the MGS
\[
	S_3,S_1,I_2,Z_a,P_1,Z_b.
\]
$Z_b$ is stable because $S_1$ comes before $P_1$. $Z_a$ is stable because $S_1$ comes before $I_2$ and the quotient splitting object $Y=S_3$ came earlier.
}

}


{
\subsection{Extension ghosts}\label{ss: extension ghosts}

We showed in \cite{GrInvRedo} how the subobject and quotient object ghosts help to calculate the ``generalized Grassmann invariant''
\[
	\chi:K_3 \ZZ[\pi]\to H_0(\pi;\ZZ_2[\pi]).
\]
Although we already had an explicit formula for this invariant in \cite{GrInvRedo}, the idea for the ghosts is that we can just draw pictures and avoid explicit calculations. With a view towards this application, the key property of ghosts is the bifurcation formula. In all of our examples, the ghosts are created, they bifurcate, then they die. We only get a nontrivial invariant of the picture if ghosts ``escape''. 
}

\begin{defn}\label{def: extension ghost}
Given an extension $A\to B\to C$ where $A,B,C$ are all bricks in a class of objects $\cL$ which may not be extension closed. Then the domain of the \emph{extension ghost} $\widetilde B=\cG h(A\to B\to C)$ is given by $D(\widetilde B)=\{\theta\,:\, \theta(B)=0, \theta(C)\le0\}$. Recall that $D(B)$ is given by the conditions $\theta(B)=0$ and $\theta(Y)\ge0$ for all weakly admissible quotients $Y$ of $B$. This includes $C$ since $B\to C$ is an admissible epimorphism. We say that $\widetilde B$ is a \emph{minimal extension ghost} if $A,C$ are minimal (have no weakly admissible quotients) and $C$ is the only weakly admissible quotient of $B$. (See Figures \ref{Figure11} and \ref{Figure12}.)
\end{defn}

{
The new extension ghosts are a very promising innovation. We hope that they can help to compute the higher Reidemeister torsion invariant of pictures. This will be explained further in the next paper. For now we just give two example. The first example (Figure \ref{Figure11}) shows that, if the same module can be expressed as an extension in two different way, it produces two extension ghosts which will invariably overlap. Figure \ref{Figure13} shows that two ``extension clones'' might be created. This will be explained in \cite{MoreGhosts2}. The second example (Figure \ref{Figure12}) includes the two ghosts from the first example and adds two more ghosts. The new ghosts are ``minimal extension ghosts'' since the middle term in $A\to B\to C$ can be expresses as an extension in only one way and the two objects $A,C$ are minimal, with no admissible quotients.
}

{%
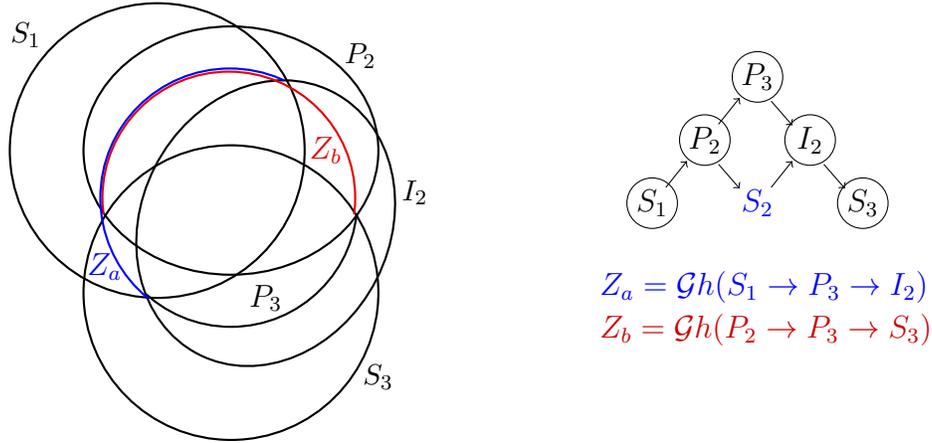
\begin{figure}[htbp]
\begin{center}

\begin{tikzpicture}[scale=1.4] 
%
{
\begin{scope}

\begin{scope}
\draw[thick] (0,0) ellipse [x radius=1.4cm,y radius=1.18cm];
\end{scope}

\begin{scope}[xshift=3.3mm,yshift=-6.9mm]
\begin{scope}[rotate=63]
\draw[thick] (0,0) ellipse [x radius=1.4cm,y radius=1.18cm];
\end{scope}
\end{scope}

\begin{scope}[xshift=.87mm]
		\draw[thick] (1.4,.9) node[left]{\small$P_2$}; 
		\draw[thick] (1.66,-.4) node{\small$I_2$}; 
		
\end{scope}
\begin{scope}[xshift=-.7cm]
	\draw[thick] (0,0) circle[radius=1.4cm];
		\draw (-1,1.1) node[left]{\small$S_1$}; 
\end{scope}
%
\begin{scope}[yshift=-1.35cm]
	\draw[thick] (0,0) circle[radius=1.4cm];
	\draw (1.15,-.8) node[right]{\small$S_3$}; 
\end{scope}
%
\coordinate(B1) at (-1.39,-1.22);
\coordinate(B2) at (.69,-.14);
\coordinate(B) at (-.9,-.95);
\coordinate(A) at (1.4,.16);
\coordinate(C) at (-.48,-.72);
\coordinate(AC) at (-0.05,-.4);
\end{scope} 

\begin{scope}
\clip (-.8,-.6) rectangle (1.3,-2.3);
\draw[thick] (0,-.475) circle[radius=12mm];
\end{scope}

\coordinate(Za) at (-1.2,-1.1);
\coordinate(Zb) at (.9,0);
\coordinate(P3) at (.33,-1.4);

\draw[red!90!black] (Zb)node{$Z_b$};
\draw (P3)node{$P_3$};
\draw[blue] (Za)node{$Z_a$};

\begin{scope}
\clip (-1.3,-.6) rectangle (1.3,1.3);
\draw[thick, red!90!black] (-.02,-.45) circle[radius=12mm];
\end{scope}
\begin{scope}
\clip (-1.5,-1.4) rectangle (.5,.9);
\draw[thick, blue] (-.01,-.45) circle[radius=12.3mm];
\end{scope}
}

{ 
\begin{scope}[xshift=4cm,yshift=-.5cm]
\draw (0,0) node{$S_1$}; 
\draw(0,0)circle[radius=2.4mm];
\draw(0.5,0.6)circle[radius=2.4mm];
\draw (0.5,0.6) node{$P_2$}; 
\draw (1.5,0.6) node{$I_2$};
\draw(1.5,0.6)circle[radius=2.4mm];
\draw (1,1.2) node{$P_3$};
\draw(1,1.2)circle[radius=2.4mm];
\draw[blue] (1,0) node{$S_2$}; 
\draw (2,0) node{$S_3$}; 
\draw(2,0)circle[radius=2.4mm];
\draw[->] (0.13,0.15)--(.33,.4);
\draw[->] (1.13,0.15)--(1.33,0.4);
\draw[->] (0.63,0.75)--(.83,1);
\draw[->] (0.63,0.37)--(.83,.14);
\draw[->] (1.63,.37)--(1.83,.14);
\draw[->] (1.13,.97)--(1.33,.74);
\end{scope}
\begin{scope}[xshift=3.4cm,yshift=-1.3cm]
\draw[blue] (0,0)node[right]{$Z_a=\cG h(S_1\to P_3\to I_2)$};
\draw[red!80!black] (0,-.4)node[right]{$Z_b=\cG h(P_2\to P_3\to S_3)$};
\end{scope}
}
\end{tikzpicture}

\caption{Since the module $P_3$ is an extension in two ways: as $S_1\to P_3\to I_2$ and $P_2\to P_3\to S_3$, it produces two extension ghosts. These are indicated by red and blue curves which overlap. {\color{blue}$Z_a$} is outside the $I_2$ circle, {\color{red}$Z_b$} is outside the $S_3$ circles. $D(P_3)$ is inside both of these circles.}
\label{Figure11}
\end{center}
\end{figure}
}

{%
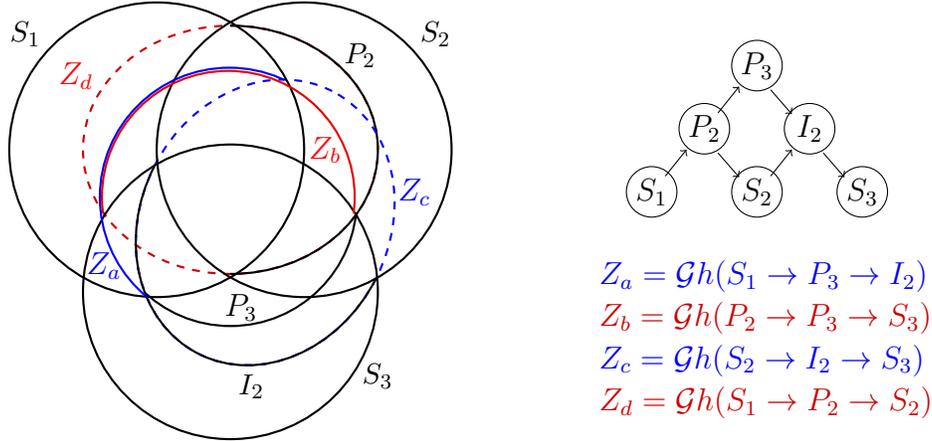
\begin{figure}[htbp]
\begin{center}

\begin{tikzpicture}[scale=1.4] 
%
{
\begin{scope}
\begin{scope}
\draw[thick, red!80!black,dashed] (0,0) ellipse [x radius=1.4cm,y radius=1.18cm];
\clip rectangle (2,1.3) rectangle (0,-1.3);
\draw[thick] (0,0) ellipse [x radius=1.4cm,y radius=1.18cm];
\end{scope}

\begin{scope}[xshift=3.3mm,yshift=-6.9mm]
\begin{scope}[rotate=63]
\draw[thick, blue, dashed] (0,0) ellipse [x radius=1.4cm,y radius=1.18cm];
\clip rectangle (-2,-1.3) rectangle (0,1.3);
\draw[thick] (0,0) ellipse [x radius=1.4cm,y radius=1.18cm];
\end{scope}
\end{scope}

\begin{scope}[xshift=.87mm]
		\draw[thick] (1.4,.9) node[left]{\small$P_2$}; 
		\draw[thick,blue] (1.66,-.4) node{$Z_c$}; 
\end{scope}

\begin{scope}[xshift=-.7cm]
	\draw[thick] (0,0) circle[radius=1.4cm];
		\draw (-1,1.1) node[left]{\small$S_1$}; 
\end{scope}
\begin{scope}[xshift=.7cm]
	\draw[thick] (0,0) circle[radius=1.4cm];
	\draw (1,1.1) node[right]{\small$S_2$};
\end{scope}
\begin{scope}[yshift=-1.35cm]
	\draw[thick] (0,0) circle[radius=1.4cm];
	\draw (1.15,-.8) node[right]{\small$S_3$}; 
\end{scope}
%
\coordinate(B1) at (-1.39,-1.22);
\coordinate(B2) at (.69,-.14);
\coordinate(B) at (-.9,-.95);
\coordinate(A) at (1.4,.16);
\coordinate(C) at (-.48,-.72);
\coordinate(AC) at (-0.05,-.4);
\end{scope} 
}
\coordinate(Za) at (-1.2,-1.1);
\coordinate(Zb) at (.9,0);
\coordinate(P3) at (.33,-1.4);
\coordinate(I2) at (.2,-2.25);

\draw[red] (Zb) node{$Z_b$};
\draw (.1,-1.5) node{$P_3$};
\draw[red] (-1.2,.7) node[left]{$Z_d$};
\draw (I2) node{\small$I_2$};

\begin{scope} 
\clip (-1.5,-1.4) rectangle (.5,.9);
\draw[thick, blue] (-.01,-.45) circle[radius=12.3mm];
\end{scope}

\begin{scope}
\clip (-1.3,-.6) rectangle (1.3,1.3);
\draw[thick, red!90!black] (-.02,-.45) circle[radius=12mm];
\end{scope}

\begin{scope} 
\clip (-.8,-.6) rectangle (1.3,-2.3);
\draw[thick] (0,-.475) circle[radius=12mm];
\end{scope}


\draw[blue] (Za) node{$Z_a$};

{ 
\begin{scope}[xshift=4cm,yshift=-.4cm]
\draw (0,0) node{$S_1$}; 
\draw(0,0)circle[radius=2.4mm];
\draw(0.5,0.6)circle[radius=2.4mm];
\draw (0.5,0.6) node{$P_2$}; 
\draw (1.5,0.6) node{$I_2$};
\draw(1.5,0.6)circle[radius=2.4mm];
\draw (1,1.2) node{$P_3$};
\draw(1,1.2)circle[radius=2.4mm];
\draw (1,0) node{$S_2$}; 
\draw(1,0)circle[radius=2.4mm];
\draw (2,0) node{$S_3$}; 
\draw(2,0)circle[radius=2.4mm];
\draw[->] (0.13,0.15)--(.33,.4);
\draw[->] (1.13,0.15)--(1.33,0.4);
\draw[->] (0.63,0.75)--(.83,1);
\draw[->] (0.63,0.37)--(.83,.14);
\draw[->] (1.63,.37)--(1.83,.14);
\draw[->] (1.13,.97)--(1.33,.74);
\end{scope}
\begin{scope}[xshift=3.4cm,yshift=-1.2cm]
\draw[blue] (0,0)node[right]{$Z_a=\cG h(S_1\to P_3\to I_2)$};
\draw[red!80!black] (0,-.4)node[right]{$Z_b=\cG h(P_2\to P_3\to S_3)$};
\draw[blue] (0,-.8)node[right]{$Z_c=\cG h(S_2\to I_2\to S_3)$};
\draw[red!80!black] (0,-1.2)node[right]{$Z_d=\cG h(S_1\to P_2\to S_2)$};
\end{scope}
}
\end{tikzpicture}

\caption{Here there are 4 extension ghosts. The two from Figure \ref{Figure11}: $\color{blue}Z_a=\cG h(S_1\to P_3\to I_2)$ and $\color{red!80!black}Z_b=\cG h(P_2\to P_3\to S_3)$ and two new extension ghosts $\color{blue}Z_c=\cG h(S_2\to I_2\to S_3)$ and $\color{red!80!black}Z_d=\cG h(S_1\to P_2\to S_2)$. The new extension ghosts are minimal extension ghosts. In this figure, we see two bifurcations: $Z_a$ splits off of $Z_c$ when it hits the splitting wall $D(S_1)$ and $Z_b$ splits off of $Z_d$ at the splitting wall $D(S_3)$.}
\label{Figure12} 
\end{center}
\end{figure}
}

{
The example in Figure \ref{Figure11} also has a subobject ghost and a quotient object ghost shown in Figure \ref{Figure10}. Since they are ghosts of the same object $S_2$, their domains overlap.

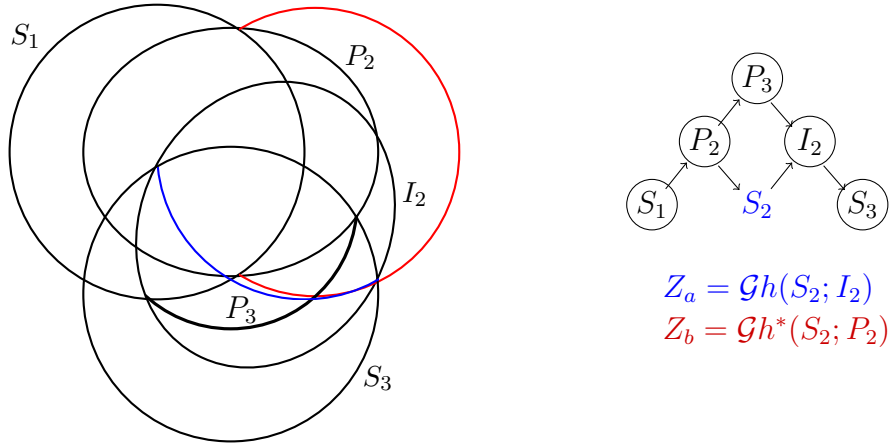
\begin{figure}[htbp]
\begin{center}

\begin{tikzpicture}[scale=1.4] 
%
{
\begin{scope}

\begin{scope}
\draw[thick] (0,0) ellipse [x radius=1.4cm,y radius=1.18cm];
\end{scope}

\begin{scope}[xshift=3.3mm,yshift=-6.9mm]
\begin{scope}[rotate=63]
\draw[thick] (0,0) ellipse [x radius=1.4cm,y radius=1.18cm];
\end{scope}
\end{scope}

\begin{scope}[xshift=.87mm]
		\draw[thick] (1.4,.9) node[left]{\small$P_2$}; 
		\draw[thick] (1.66,-.4) node{\small$I_2$}; 
		
\end{scope}
\begin{scope}[xshift=-.7cm]
	\draw[thick] (0,0) circle[radius=1.4cm];
		\draw (-1,1.1) node[left]{\small$S_1$}; 
\end{scope}
\begin{scope}[xshift=.7cm]
\clip (-.62,-1.5)rectangle (1.5,1.5);
	\draw[thick,red] (0.1,0) circle[radius=1.37cm];
\end{scope}
\begin{scope}[yshift=-1.35cm]
	\draw[thick] (0,0) circle[radius=1.4cm];
	\draw (1.15,-.8) node[right]{\small$S_3$}; 
\end{scope}
%
\coordinate(B1) at (-1.39,-1.22);
\coordinate(B2) at (.69,-.14);
\coordinate(B) at (-.9,-.95);
\coordinate(A) at (1.4,.16);
\coordinate(C) at (-.48,-.72);
\coordinate(AC) at (-0.05,-.4);
\end{scope} 
}

\draw (.1,-1.5) node{$P_3$};

\begin{scope}
\clip (-.8,-.6) rectangle (1.3,-2.3);
\draw[very thick] (0,-.48) circle[radius=12mm];
\end{scope}

\begin{scope}
\clip (-.7,-.15)rectangle (1.4,-1.4);
\draw[thick,blue] (0.7,0.0) circle[radius=1.4cm];
\end{scope}

{ 
\begin{scope}[xshift=4cm,yshift=-.5cm]
\draw (0,0) node{$S_1$}; 
\draw(0,0)circle[radius=2.4mm];
\draw(0.5,0.6)circle[radius=2.4mm];
\draw (0.5,0.6) node{$P_2$}; 
\draw (1.5,0.6) node{$I_2$};
\draw(1.5,0.6)circle[radius=2.4mm];
\draw (1,1.2) node{$P_3$};
\draw(1,1.2)circle[radius=2.4mm];
\draw[blue] (1,0) node{$S_2$}; 
\draw (2,0) node{$S_3$}; 
\draw(2,0)circle[radius=2.4mm];
\draw[->] (0.13,0.15)--(.33,.4);
\draw[->] (1.13,0.15)--(1.33,0.4);
\draw[->] (0.63,0.75)--(.83,1);
\draw[->] (0.63,0.37)--(.83,.14);
\draw[->] (1.63,.37)--(1.83,.14);
\draw[->] (1.13,.97)--(1.33,.74);
\end{scope}
\begin{scope}[xshift=4cm,yshift=-1.3cm]
\draw[blue] (0,0)node[right]{$Z_a=\cG h(S_2;I_2)$};
\draw[red!80!black] (0,-.4)node[right]{$Z_b=\cG h^\ast(S_2;P_2)$};

\end{scope}
}
\end{tikzpicture}

\caption{Figure \ref{Figure11} also has a subobject ghost and quotient object ghost shown above. Thus, this example has all three kinds of ghosts.}
\label{Figure10}
\end{center}
\end{figure}

}

{%
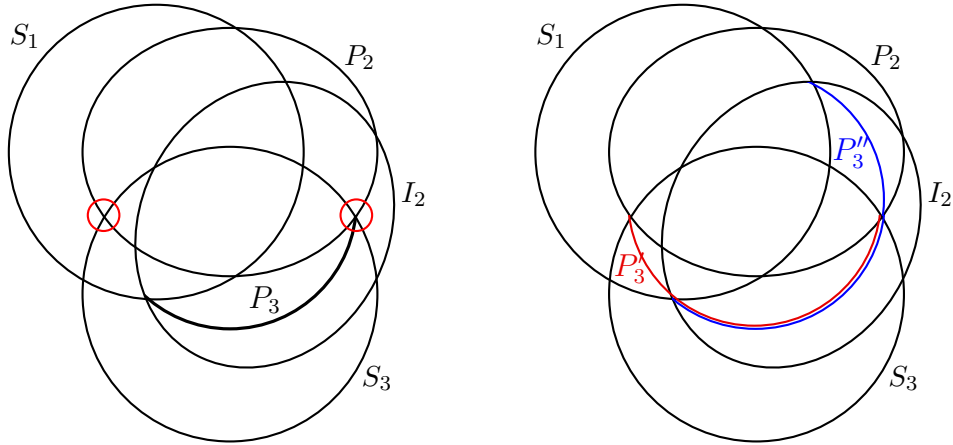
\begin{figure}[htbp]
\begin{center}

\begin{tikzpicture}[scale=1.4] 
%
{
\begin{scope}

\begin{scope}
\draw[thick] (0,0) ellipse [x radius=1.4cm,y radius=1.18cm];
\end{scope}

\begin{scope}[xshift=3.3mm,yshift=-6.9mm]
\begin{scope}[rotate=63]
\draw[thick] (0,0) ellipse [x radius=1.4cm,y radius=1.18cm];
\end{scope}
\end{scope}

\begin{scope}[xshift=.87mm]
		\draw[thick] (1.4,.9) node[left]{\small$P_2$}; 
		\draw[thick] (1.66,-.4) node{\small$I_2$}; 
		
\end{scope}
\begin{scope}[xshift=-.7cm]
	\draw[thick] (0,0) circle[radius=1.4cm];
		\draw (-1,1.1) node[left]{\small$S_1$}; 
\end{scope}
%
\begin{scope}[yshift=-1.35cm]
	\draw[thick] (0,0) circle[radius=1.4cm];
	\draw (1.15,-.8) node[right]{\small$S_3$}; 
\end{scope}
%
\coordinate(B1) at (-1.39,-1.22);
\coordinate(B2) at (.69,-.14);
\coordinate(B) at (-.9,-.95);
\coordinate(A) at (1.4,.16);
\coordinate(C) at (-.48,-.72);
\coordinate(AC) at (-0.05,-.4);
\end{scope} 

\begin{scope}
\clip (-.8,-.6) rectangle (1.3,-2.3);
\end{scope}


\draw[red!90!black] (-1.2,-1.1)node{$P_3'$};
\draw[blue] (.9,0)node{$P_3''$};

\begin{scope}
\clip (-1.3,-.6) rectangle (1.3,-2.3);
\draw[thick, red!90!black] (-.02,-.45) circle[radius=12mm];
\end{scope}
\begin{scope}
\clip (-.8,-1.8) rectangle (1.3,-.6);
\draw[thick, blue] (-.01,-.45) circle[radius=12.3mm];
\end{scope}

\begin{scope}
\clip (0,.66) rectangle (1.3,-.6);
\draw[thick, blue] (-.02,-.45) circle[radius=12.3mm];
\end{scope}
}
{
\begin{scope}[xshift=-5cm] 

\draw (.33,-1.4)node{$P_3$};

\begin{scope}
\clip (-.8,-.6) rectangle (1.3,-2.3);
\draw[very thick] (0,-.48) circle[radius=12mm];
\end{scope}

\begin{scope}
\draw[thick] (0,0) ellipse [x radius=1.4cm,y radius=1.18cm];
\end{scope}

\begin{scope}[xshift=3.3mm,yshift=-6.9mm]
\begin{scope}[rotate=63]
\draw[thick] (0,0) ellipse [x radius=1.4cm,y radius=1.18cm];
\end{scope}
\end{scope}

\begin{scope}[xshift=.87mm]
		\draw[thick] (1.4,.9) node[left]{\small$P_2$}; 
		\draw[thick] (1.66,-.4) node{\small$I_2$}; 
\end{scope}
\begin{scope}[xshift=-.7cm]
	\draw[thick] (0,0) circle[radius=1.4cm];
		\draw (-1,1.1) node[left]{\small$S_1$}; 
\end{scope}
%
\begin{scope}[yshift=-1.35cm]
	\draw[thick] (0,0) circle[radius=1.4cm];
	\draw (1.15,-.8) node[right]{\small$S_3$}; 
\end{scope}
%
\draw[thick,red] (-1.2,-.6) circle[radius=1.5mm]; 
\draw[thick,red] (1.2,-.6) circle[radius=1.5mm]; 
\coordinate(B1) at (-1.39,-1.22);
\coordinate(B2) at (.69,-.14);
\coordinate(B) at (-.9,-.95);
\coordinate(A) at (1.4,.16);
\coordinate(C) at (-.48,-.72);
\coordinate(AC) at (-0.05,-.4);
\end{scope} 
}

\end{tikzpicture}

\caption{On the left is the relative stability diagram for the example in Figure \ref{Figure11}. This diagram does not satisfy the definition of a ``picture'' as given in \cite{ITW}, \cite{IT14} since, e.g., $S_3$ and $P_2$ commute at the left circled crossing but these walls do not commute at the right circled crossing. We need this to be a true picture. So, we need to add two ``extension clones'' as shown on the right. The clones are indicated by red and blue curves which overlap. Thus $D(P_3)$ is doubled! More will be explained in \cite{MoreGhosts2}.
}
\label{Figure13}
\end{center}
\end{figure}
}

{
\section{Plans for the future}\label{ss: plans}

\begin{enumerate}
\item With the help of several coauthors, the next paper \cite{MoreGhosts2} will explain the relation of ghosts to algebraic K-theory. Extension ghosts are used to give new formulas for algebraic $K$-theory invariants and the formulas include exchange points. Mike Sullivan calls these ``exchange ghosts''. We also hope to extend these extension ghost computations to the higher Reidemeister torsion invariant of \cite{IgKlein}.
\item We are planning, with several coauthors, to examine how ghosts change when we rotate the heart of the derived category. This was suggested by Yun Shi after my lecture at Brandeis. In the hereditary case, Ray Maresca showed that some ghosts will become complexes (not stalk complexes) under rotation of the heart.
\item We are working with Emre Sen on ghost modules for $n$-cluster-tilting objects.
\end{enumerate}
}

{
\section*{Acknowledgments}
The author would like to thank Gordana Todorov for several suggestions about this paper. Without her help and support, this project and many others would not have been completed. Also, Mike Sullivan, Ray Maresca and Yun Shi helped to shape this paper and to make plans for sequels with comments and suggestions. We thank Emre Sen for pointing out mistakes in the figures for extension ghosts in version 1. The author is also grateful to the Simons Foundation for its support: Grant \#686616.
}

\end{document}